\documentclass[a4paper,11pt,reqno]{article}
\usepackage{color}
\usepackage{epsfig,psfrag,amsmath,amssymb,latexsym}
\usepackage{amsthm}
\usepackage{amscd}
\usepackage{amsfonts}
\usepackage{enumerate}
\usepackage{graphicx}
\usepackage{bbm}
\usepackage{mathrsfs}
\usepackage{mdframed}
\usepackage{relsize}
\usepackage{hyperref}

\newtheorem{remark}{Remark}[section]

\newcommand{\NN}{\mathbb{N}}
\newcommand{\RR}{\mathbb{R}}

\newcommand{\PP}{\mathbb{P}}

\newcommand{\EE}{\mathbb{E}}

\newcommand{\CC}{\mathbb{C}}

\newcommand{\T}{\mathbb{T}}
\newcommand{\m}{{\,\mathfrak{m}}}
\newcommand{\M}{{\,\mathfrak{M}}}

\newcommand{\FF}{\mathbb{F}}

\newcommand{\HH}{\mathbb{H}}
\newcommand{\LL}{{\cal L}}

\newcommand{\1}{{\bf 1}}
\newcommand{\ind}{\mathbbm{1}}

\newcommand{\q}{q}

\def\be{\begin{eqnarray}}
\def\ee{\end{eqnarray}}
\def\ben{\begin{eqnarray*}}
\def\een{\end{eqnarray*}}
\numberwithin{equation}{section}

\def\bi{\bigskip \noindent}
\def\me{\medskip \noindent}

\newtheorem{proposition}{Proposition}[section]

\newtheorem{lemma}[proposition]{Lemma}
\newtheorem{theorem}[proposition]{Theorem}
\newtheorem{corollary}[proposition]{Corollary}

\begin{document}

\date{\today\\fichier: \jobname}

\title{\textbf{Diffusions from Infinity}}

\author{Vincent Bansaye\thanks{CMAP, Ecole Polytechnique, CNRS, route de
   Saclay, 91128 Palaiseau Cedex-France; vincent.bansaye@polytechnique.edu},
    Pierre Collet\thanks{CPHT, Ecole Polytechnique, CNRS, route de
   Saclay, 91128 Palaiseau Cedex-France; pierre.collet@cpht.polytechnique.fr},  Servet Martinez\thanks{CMM-DIM;  Universidad de Chile; UMI-CNRS 2807; 
Casilla 170-3 Correo 3 Santiago; Chile; smartine@dim.uchile.cl},\\ 
Sylvie M\'el\'eard\thanks{CMAP, Ecole Polytechnique, CNRS, route de
Saclay, 91128 Palaiseau Cedex-France; sylvie.meleard@polytechnique.edu}, Jaime San Martin\thanks{CMM-DIM;  Universidad de Chile; UMI-CNRS 2807; 
Casilla 170-3 Correo 3 Santiago; Chile; jsanmart@dim.uchile.cl}}

\maketitle

\begin{abstract} In this paper we consider diffusions on the half line
  $(0,\infty)$ such that the
expectation of the arrival time at the origin is uniformly
bounded in the initial point. This implies that there is a well
defined diffusion process starting from infinity, which takes finite values at positive times. We 
study the behaviour of  hitting times of large barriers and in a dual way, the behaviour of  the process starting at infinity
for small time. In particular we prove that the process coming down from infinity is in
small time governed by  a specific deterministic function. 
Suitably normalized fluctuations of the hitting times are asymptotically
Gaussian.  We also derive the tail of the distribution of the
hitting time of the origin and a Yaglom limit for the diffusion starting from infinity. 
We finally prove that the distribution of this process killed at the origin is absolutely continuous with respect
to the speed measure. The density is expressed in terms of the
eigenvalues and eigenfunctions of the generator of the killed diffusion.
\end{abstract}



\section{Introduction and main results}
In this paper we are interested in the descent from infinity for continuous diffusion processes without explosion.
This article is also a refinement about the study of quasi-stationary distributions (q.s.d.) when $\infty$ is an entrance boundary.
Although, there is a large literature of q.s.d. for one dimensional diffusions killed at $0$, there are not many
results on the behaviour of these processes near infinity and  the main
works we use in this article can be found in  \cite{Cattiaux}, \cite{S_E(2007)}, \cite{Littin} and \cite{K_S(2012)}.
On the other hand, this article was inspired by the results in \cite{B_M_M} in the context of birth-and-death
processes and is linked to the approximation by flows coming from infinity in \cite{B}.

\medskip

For convenience, we consider diffusion processes which are stopped at a regular point chosen to be the origin. 
More precisely we consider the flow $X$ of diffusions on $\RR^+$ stopped at $0$ satisfying for any $x\in \mathbb{R}_{+}$
$$
dX^x_t=dB_t-\q(X_t^x) dt\ ,\  X^x_0=x,
$$
where $B$ is a standard B.M. and $\q$ is assumed to be $\mathcal{C}^1([0,\infty))$. 
In particular $0$ is a regular boundary point, as
well as  any $z\in \RR^+$, when considering the diffusion on $[z,\infty)$. In our exposition $0$ will play no fundamental role
and most of the results only depend on the behaviour of $q$ near infinity.
Without loss of generality
we assume that $B$ is the coordinate process in the canonical space $(\mathcal{C}([0,\infty)),\mathscr{F},\FF,\PP)$, where
$\PP$ is the Wiener measure. We also denote by $\EE$ the associated expectation. Similarly, we denote by
$\PP_x$ the distribution of the process $X^x$ and $\EE_x$ its associated expectation.

\smallskip
The hitting time of a point $z$ by the process $X^x$ is denoted by  $\T_z^{(x)}=\inf\{t\ge 0: X^x_t= z\}$. 
To avoid overly burdensome notation, when the initial condition $x$ of the diffusion $X$ is clear, we use $\T_z=\T^{(x)}_z$. This will be
of particular use when computing moments of $\T^{(x)}_z$.

The scale function $\Lambda$  of $X$ is given by 
$$
\Lambda(z)=\int_0^z e^{\gamma(y)}\, dy, \hbox{ where }
\gamma(y)=2\int_0^y \q(\xi) \, d\xi.
$$
Throughout the paper  we shall assume the following main hypothesis

\smallskip

$$  \HH_1: \qquad \qquad  \qquad \qquad   
%$: The following integral is finite 
%\be\label{H1}
\int_0^\infty e^{\gamma(y)} \int_y^\infty e^{-\gamma(z)}\, dz\,dy<\infty.  \qquad \qquad  \qquad \qquad  $$
%\ee
We also introduce the associated Lyapunov function
\be
 \label{lyapounov}
  \m(z):=2\int_z^\infty e^{\gamma(y)} \int_y^\infty e^{-\gamma(\xi)} d\xi\,dy\,,
\ee
which is positive and decreasing and it will play a major role in what follows.

\me The speed measure $\mu$ is given by  $d\mu(y)=2e^{-\gamma(y)} dy$.  
By $\HH_1$,  $\int_y^\infty e^{-\gamma(z)}\, dz<\infty$ for $dy$-a.e., which implies that 
$\int_y^\infty e^{-\gamma(z)}\, dz<\infty$ for all $y>0$. Therefore 
 $\mu$ is a finite measure. 
Moreover,  the following  inequality
$$
(x-1)^2=\left(\int_1^x e^{ \gamma(y)/2}e^{- \gamma(y)/2}\, dy\right)^2\le \int_1^x e^{\gamma(y)}\, dy \int_1^x e^{-\gamma(y)}\, dy
$$   leads  immediately to $\Lambda(\infty)=\infty$. That implies (see
\cite{IW}) that the process  does not explode and attains $0$ almost surely.

\bi
The next result gives properties equivalent to  $\HH_1$ and it will be proven in Section 2. 
\smallskip

\begin{proposition}
\label{pro:0}
Assume that $\q$ is $\mathcal{C}^1([0,\infty))$, then the following are equi\-valent

\noindent$(i)$ \hspace{0.15cm} $\HH_1$ holds,

\noindent$(ii)$ \hspace{0.05cm} $\sup\limits_{x\ge 0} \EE(\T_0^{(x)})<\infty$,

\noindent$(iii)$ there exists $a>0$ such that  $\sup\limits_{x\ge 0} \EE(e^{\lambda \T_0^{(x)}})<\infty$, for any $\lambda\le a$.

\noindent$(iv)$ $\;\infty$ is an entrance boundary, that is, there exists $y\ge 0, t>0$ such that
$$
\lim\limits_{x\to \infty} \PP(\T_y^{(x)}<t)>0.
$$
\end{proposition}

\bi In Section \ref{sec:P_infty}, we  define $X^{\infty}$ by monotonicity of the flow and  prove that  when $x\to \infty$, the process $X^x$ converges a.s., 
uniformly on compact sets of $(0,\infty)$ (in time) to $X^\infty$ that satisfies: for all $0<s<t\le \T_0^{(\infty)}=\lim\limits_{x\to\infty} \T^{(x)}_0$,
$$
X^\infty_t=X^\infty_s+B_t-B_s-\int_s^t \q(X^\infty_u) \, du, \, X^\infty_0=\infty.
$$
We denote by $\PP_\infty$ the distribution of $X^\infty$ and $\EE_\infty$ the associated expectation. Similarly, we define
$\T_z^{(\infty)}=\lim\limits_{x\to\infty} \T^{(x)}_z$, which results on the hitting time of $z$ for the process $X^\infty$.

We will quantify the speed of convergence using the Lyapunov function $\m$. 
The process $X^{\infty}$ provides  a relevant approximation of the diffusion when  the initial state is large. 
In particular, it yields   the large population approximation of diffusion processes where the 
demographic parameters are fixed but the initial size of the population is large. As in \cite{B}, 
we consider the image of the flow under a diffeomorphism  which gives a natural distance to 
compare the stochastic flow with its limit at infinity and a deterministic function. But here we can 
exploit the Lyapunov function $\m$ and its regularity, which provides a relevant and tractable 
compactification of the space.

\bi
The spectral study of the semigroup of $X$ as well as a fine study of the behavior of this process near $\infty$ 
will be done under an extra assumption on the asymptotic behavior of $\q$.
\smallskip
$$ \HH_2: \qquad \qquad 
%\noindent${\bf \HH_2}$: $\q$ has the following behavior at $\infty$
\lim\limits_{x\to \infty} \q(x)=\infty \quad \text{and} \quad
\lim\limits_{x\to\infty} \frac{\q'(x)}{\q^2(x)}=0. \qquad \qquad \qquad
$$

An important and interesting remark is that under $\HH_2$ it holds
$$
2q(y)e^{\gamma(y)}\int_y^\infty e^{-\gamma(\xi)} \,d\xi\,
{\underset{y \to \infty}{\longrightarrow}} 1,
$$
which is a direct consequence of l'H\^opital's rule, see Corollary \ref{cor:1}. 
In particular, under $\HH_1$ and $\HH_2$, we have  for large $z$ 
\be
\label{convint}
\int_z^\infty \frac{1}{\q(x)}\, dx<\infty.
\ee
Noting  that 
\be
\label{eq:Mgrande}
M(z)=\int_z^\infty \frac{1}{\q(y)} \, dy
\ee
is the time it takes for the deterministic flow ($\dot y=-q(y)$) 
starting at infinity to reach the point $z$, we have that 
$\m(z)=\EE(\T_z^{(\infty)})$ is asymptotically equivalent to $M(z)$ (see Corollary \ref{cor:1} $(ii)$) and 
\be
\label{eq:m_M}
\lim\limits_{z \to \infty} \frac{\m(z)}{M(z)}=1.
\ee

\bi We can now state the main results of this article. First, we describe the behavior of the hitting of large integers and deduce an approximation of the process starting from infinity for small times. We follow \cite{B_M_M}
in the gradual regime and prove a Law of Large Numbers (LLG)  and a Central Limit Theorem (CLT). The main difference is the fact that the state space is continuous, which makes the study  of the moments of hitting times and the derivation of the position of the process more involved. On the other hand, we are able here to make all the assumptions explicit in terms of $q$ using finer estimates.

\smallskip

\begin{theorem}  
\label{the:1} Assume that $ \HH_1$ and  $\HH_2$ hold, then
 in probability
$$
\lim\limits_{z\to \infty} \,\frac{\T^{(\infty)}_z}{\EE(\T^{(\infty)}_z)}=1.
$$
If we assume further that 

\be
\label{hypotruc}
\int_{z_0}^\infty \frac{1}{\q^3(y) (\int_y^\infty \q^{-1}(x) \, dx)^2}\, dy<\infty,
\ee
then the convergence is almost sure.
\end{theorem}

\medskip

\bi{\bf Notation:}  
We use the short hand notation $f \approx g$ to mean that 
$0< \liminf\limits_{z\to \infty} \frac{f(z)}{g(z)}\le \limsup\limits_{z\to \infty} \frac{f(z)}{g(z)}<\infty$.

\begin{remark} The hypothesis \eqref{hypotruc}
is satisfied in the following cases
\begin{itemize}
\item $\q(x)=e^{p(x)}$, with $p(x)\to \infty$, $p'(x)e^{-p(x)}\to 0$, $p''(x)/(p'(x))^2\to 0$ as $x\to \infty$
and $\int_{z_0}^\infty  [(p'(x))^2\vee 1]\,  e^{-p(x)} \, dx<\infty$. Indeed, consider $H=e^{-p}$, then
$$
H'=-p'H,\, H''=-p''H+(p')^2H,
$$
and
$$
\frac{HH''}{(H')^2}=\frac{(-p''+(p')^2)H^2}{(p')^2H^2}\to 1 \hbox{ at } \infty.
$$
So, by Lemma \ref{lem:1}, we get 
$$
\int_x^\infty \frac{1}{\q(y)}\, dy\approx \frac{H^2(x)}{-H'(x)}=\frac{e^{-p(x)}}{p'(x)},
$$
and \eqref{hypotruc} follows from $\int_{z_0}^\infty  (p'(x))^2 e^{-p(x)} \, dx<\infty$. 

\me
\item $\q(x)\approx x^{a}$ for $a>1$. It is clear that $\int_y^\infty \frac{1}{\q(x)}\, dx\approx \frac{1}{a-1} y^{1-a}$
so
$$
\int_{z_0}^\infty \frac{1}{\q^3(y) \left(\int_y^\infty \q^{-1}(x) \, dx\right)^2} dy 
\approx (a-1)^2\int_{z_0}^\infty \frac{1}{y^{a+2}} \, dy<\infty
$$

\end{itemize}
\end{remark}

\medskip

\begin{corollary} 
\label{cor:Pierre}
Assume that $ \HH_1$ and  $\HH_2$ hold, then the following limits exist
\begin{itemize} 
\item[(i)]  For all $\lambda<1$, we have
$$
\lim\limits_{z\to \infty} \EE_\infty\left(\exp\left(\lambda\, \frac{\T_z}{\EE_\infty(\T_z)}\right)\right)=e^\lambda
$$
\item[(ii)] For all $n\ge 1$,
$$
\lim\limits_{z\to \infty} \frac{\EE_\infty(\T_z^{n})}{(\EE_\infty(\T_z))^{n}}=1.
$$
\end{itemize}
\end{corollary}

\medskip

\begin{theorem}
\label{the:2}
Assume that $ \HH_1$ and  $\HH_2$ hold. Then, we have the limit in distribution
$$
\frac{\T^{(\infty)}_z-\EE(\T^{(\infty)}_z)}{\sqrt{\emph{Var}(\T^{(\infty)}_z)}}\,\overset{\mathscr{D}}{\underset{z \to \infty}{\longrightarrow}}\,\mathsf{Z},
$$
where $\mathsf{Z}$ has a standard Gaussian distribution.
\end{theorem}

\me We shall see in Appendix \ref{moments}  that 
\be
\label{varex}\lim\limits_{z\to \infty} \frac{\hbox{Var}(\T^{(\infty)}_z)}{\int_z^\infty \frac{1}{\q^3(y)} \, dy}=1.
\ee

\bigskip
%In all the paper we  assume that Assumptions $ \HH_1$ and  $\HH_2$ hold.

We can now invert in a sense this result and derive the fluctuations
of the process starting from infinity.

\begin{theorem}
\label{fluctuX}
Assume  $\HH_{1}$, $\lim\limits_{x\to\infty}q(x)=\infty$,  together with
\begin{equation}
\label{eq:H2.5.1}
0<\Sigma:=\lim\limits_{z\to\infty}\frac{\int_{z}^{\infty}q(x)^{-1}\;dx}{q^{2}(z)\;\int_{z}^{\infty}q(x)^{-3}dx}<\infty\;,
\end{equation}
and
\begin{equation}
\label{eq:H2.5.2}
\lim\limits_{z\to\infty}\frac{q'(z)}{q(z)}\sqrt{\int_{z}^{\infty}\frac{1}{q(x)}\;dx}=0\;.
\end{equation}
Then $\HH_{2}$ holds and 
$$
\frac{X_{t}^{\infty}-\m^{-1}(t)}{\sqrt{\Sigma\, t}}\,\overset{\mathscr{D}}{\underset{t \searrow 0}{\longrightarrow}}\,\mathsf{Z}\;,
$$
where $\mathsf{Z}$ has a standard Gaussian distribution.
\end{theorem}

\begin{remark} 
\label{rem:remhypop}
We shall prove that under  $\HH_{1}$ if $q$ converges to $\infty$ at infinity and \eqref{eq:H2.5.2} holds, then $\HH_2$
holds (see Lemma \ref{conhypop}). On the other hand under $\HH_{1}$ and $\HH_{2}$ the hypothesis
$$
b=\lim\limits_{z\to \infty} q'(z)\int_z^\infty \frac{1}{q(x)}\, dx\in \RR,
$$
implies:  $b\ge 0$, \eqref{eq:H2.5.1} and  \eqref{eq:H2.5.2}, with $\Sigma=2b+1$.
\end{remark}

\medskip
The proofs of the theorems will be based on an extensive computation of the moments of the hitting times.
\smallskip
In what follows, we  denote by $\LL $ the second order differential operator given by
$$
\LL u=\frac12 u'' (x)-\q(x) u'(x),
$$
for all $u\in \mathcal{C}^2([0,\infty))$.
The function $\m$ defined in (\ref{lyapounov}) satisfies  
$$
 \LL \m = 1\,,
$$ 
and  by Corollary A.2 (i)  its derivative  $\m'$ is bounded on  $[0,\infty]$.

\bigskip

Until now we have studied the law  of $\T^{(\infty)}_{z}$ 
when $z$ tends to $\infty$. We now assume $z$ fixed and describe the
tail of the law of $\T^{(\infty)}_{z}$.  Our objective is the study of the domain of attraction of the unique q.s.d. including initial 
distributions that put mass at infinity. We complement this result with a spectral 
decomposition for the transition densities
of the process starting at infinity and its approximation with the densities starting at a large finite initial condition.

\medskip

To understand these results we introduce some notations and facts.
It is recalled in Appendix C (see also \cite{Cattiaux})   that under $ \HH_1$ and $ \HH_2$, the generator $L_{z}$ on
$L^2([z,\infty),\mu_z)$ of the semigroup associated with the process $X$ killed at $z$, has a discrete spectrum 
$(-\lambda_{i}(z))_{i\in \mathbb{N}}$ and the bottom of the spectrum of $-L_{z}$ is denoted by $\lambda_{1}(z)>0$. 
The associated eigenfunction $\psi_{z,1}$ is $C^2([z,\infty))$ and satisfies 
$\LL \psi_{z,1}(x)=-\lambda_{1}(z)\, \psi_{z,1}(x)$, with
the normalization  
$$
\int_z^\infty \psi_{z,1}^2(x)\, 2e^{-(\gamma(x)-\gamma(z))}\,
dx=1\,,\, \, \psi_{z,1}(z)=0\,, 
$$ 
and $\psi_{z,1}$ is positive on $(z,\infty)$. 
The next result is an extension of Theorem 5.6 in \cite{Cattiaux} (see also \cite{Littin} for the case
the drift is singular at $0$). It allows to extend the support of the initial distribution to $\infty$ in Yaglom limit.
We write $\PP_\eta$ for the probability associated to an initial condition $X_0$ distributed as $\eta$.
\medskip

\begin{theorem} 
\label{the:3}
Assume  $ \HH_1$ and $ \HH_2$ hold. Let  $z\ge 0$ and $\eta$ be any
probability measure 
on $(z,\infty]$.
\begin{itemize}
\item[(i)] Then  
$$
\lim\limits_{t\to \infty} e^{\lambda_{1}(z) t}\,  \PP_\eta(\T_z>t)=\int_z^\infty \psi_{z,1}(x) \, d\eta(x) 
\int_z^\infty \psi_{z,1}(x)\, 2e^{-(\gamma(x)-\gamma(z))}\, dx.
$$
\item[(ii)] For every Borel set $A\subset [z,\infty]$, we have 
$$
\lim\limits_{t\to \infty} \PP_\eta(X_t\in A \large | \T_z>t)=
\frac{\int_A \psi_{z,1}(x) e^{-\gamma(x)} \, dx}{\int_z^\infty \psi_{z,1}(x) e^{-\gamma(x)} \, dx}.
$$
\end{itemize}
\end{theorem}

\medskip
 
 \bi
We now  provide a formula for the density of $\PP(X^\infty_t \in dx,
\T^{(\infty)}_0>t)$ with respect to 
the speed measure $d\mu(x) =2e^{-\gamma(x)} dx$.

\medskip In \cite{Cattiaux} Theorem 2.3,  it is proved that for  any finite $z>0$,   $$\PP(X^y_t\in dx, \T^{(y)}_0>t)=r(t,y,x)\, 2e^{-\gamma(x)} dx$$
where, writing  $\lambda_k=\lambda_{k}(0)$ and $\psi_k=\psi_{0,k}$ for any integer $k\ge 1$,
$$
r(t,y,x)=\sum\limits_{k\ge 1} e^{-\lambda_k t} \psi_k(y) \psi_k(x).
$$
In order to control this sum, we require  an extra hypothesis on
$q$. This hypothesis 
prevents $q$ to have deep valleys when approaching infinity.
\medskip
$$
\HH_3: \qquad \quad  \exists a>0, x_0\ge 0  \text{ such that } \forall  x\ge x_0, \
\inf\{q(y):\, y\ge x\}\ge a\, q(x). 
$$

\medskip 
Of course we have $a\le 1$. Moreover, $a=1$ is equivalent to assume  $q$ is increasing in $[x_0,\infty)$. 
We are in position to show the following characte\-rization of $r(t,\infty,\bullet)$.

\medskip

\begin{theorem}\label{densite} Assume $\HH_1,\HH_2$ and $\HH_3$ hold. Then for all $t>0$,
\begin{itemize}
\item[(i)] as $y$ converges to $\infty$, the function $r(t,y,\bullet)=\sum_k e^{-\lambda_k t} \psi_k(y)\psi_k(\bullet)$ converges
to  $r(t,\infty,\bullet)=\sum_k e^{-\lambda_k t} \psi_k(\infty)\psi_k(\bullet)$, in all $L^p(\mu)$ for $p\ge 1$. 
In particular, $r(t,\infty,\bullet)\in L^p(\mu)$ and it is a bounded continuous function.
\item[(ii)] For any  bounded measurable function $f:\RR_+\to \RR$, we have
\begin{equation}
\label{eq:13}
\EE\left(f(X^\infty_t) ;  \T^{(\infty)}_0>t\right)=\int_0^\infty r(t,\infty,x) f(x) \, 2e^{-\gamma(x)} dx.
\end{equation}
Thus, $r(t,\infty,x)\, 2e^{-\gamma(x)} dx$ is a density for $\PP(X^\infty_t\in \bullet  \, ; \, \T^{(\infty)}_0>t)$.

\item[(iii)] For all $p\ge 1$ and all $f\in L^p$, we have
$$
\lim\limits_{z\to \infty} \EE(f(X^z_t); \T^{(z)}_0>t)=\EE(f(X^\infty_t); \T^{(\infty)}_0>t) =\int_0^\infty  r(t,\infty,x) f(x) \, 2e^{-\gamma(x)} dx.
$$
\end{itemize}
\end{theorem}

We also have  the following
strong ratio limit result. 

\medskip

\begin{theorem}\label{ratio} Assume $\HH_1, \HH_2$ and $ \HH_3$ hold, then 
$$
\lim\limits_{z\to \infty} \sup\limits_{x\ge 0} \left|\frac{r(t,\infty,x)}{r(t,z,x)}-1\right|=0.
$$
\end{theorem}

\bi
\section{The moments of $\T_z$} 
%A COMPARER AU CAS DISCRET?
\label{sec:moments}

\me Under $\HH_1$,
we shall derive a recurrence formula for the moments $\EE_x(\T_z^n)$, for all $n\ge 1$ and all $x\ge z$. 
We also develop some useful bounds for the 
moments. It is clear that, outside of a global set of $\PP$-measure $0$, the flow $X^x$ is well defined and it is
increasing in $x$, by strong uniqueness and continuity. Then for any $z$,  $\T^{(x)}_z$ is also strictly increasing in $x$, for $x\ge z$.

\medskip
\begin{theorem} 
\label{mom} Assume $\HH_1$.  

(i) We have for all $x\ge z$,
\be
\label{moment}\EE_x(\T_z) = 2\int_z^x e^{\gamma(y)}\,dy \int_y^\infty e^{-\gamma(\xi)} d\xi\ee
and the  recurrence formula: for any $n\geq 1$,
\begin{equation}
\label{for:0}
\EE_x(\T_z^n) =  2n \int_z^x e^{\gamma(y)} \int_y^\infty \EE_\xi(\T_z^{n-1}) e^{-\gamma(\xi)} d\xi.
\end{equation}
In particular, 
$$\EE_\infty(\T_z) = \m(z).$$
(ii) For all $\lambda< \frac{1}{\m(z)}$, for all $x\ge z$,
\be
\label{moment-expo}\EE_x(e^{\lambda \T_z})\le \frac{1}{1-\lambda\m(z)}.\ee
\end{theorem}

\medskip
\begin{proof}
(i) We develop a proof by induction. 

\noindent {\it Case n=1}. Consider  for any fixed $z\geq 0$ the following function
$$
u_1(x)=2\int_z^x e^{\gamma(y)} \int_y^\infty e^{-\gamma(\xi)} d\xi,
$$
which is  a $\mathcal{C}^2([0,\infty))$  positive increasing and bounded function, solution of $$
\LL u =-1, \; u(z)=0.
$$

By It\^o's formula we get
$$
u_{1}(x)-\EE_x(t\wedge \T_z)=\EE_x(u_{1}(X_{t\wedge \T_z}))=\EE_x(u_{1}(X_t)\1_{ t<\T_z}),
$$
because $u_{1}(z)=0$. The boundedness  of $u_{1}$ implies  that $\T_z$ is finite $a.s.$ and the 
  Monotone Convergence Theorem gives
$$
\EE_x(\T_z)=u_1(x)=2\int_z^x e^{\gamma(y)} \int_y^\infty e^{-\gamma(\xi)} d\xi.
$$

\medskip
\noindent {\it Case n=2}. Consider now the following function
$$
u_2(x)=2\int_z^x e^{\gamma(y)} \int_y^\infty u_1(\xi) e^{-\gamma(\xi)} d\xi,
$$
which is a $\mathcal{C}^2([0,\infty))$, positive, increasing and bounded solution to 
$\,
\LL u=-u_1,\; u(z)=0$.
 It\^o's formula gives $$
u_{2}(x)-\EE_x\left(\int_0^{t\wedge \T_z} u_1(X_s) ds\right)=\EE_x(u_{2}(X_t), t<\T_z),
$$
which converges to $0$ as $t\rightarrow \infty$ as before. Then, using $u_1(x)=\EE_x(\T_z)$,
$$u_2(x)=\int_0^{\infty } \EE_x\left(\EE_{X_s}(\T_z) \1_{s<\T_z} \right)ds=\EE_x\left(\int_0^{\T_z} (\T_z-s)\, ds\right)=\frac12\EE_x(\T_z^2)$$
since 
the strong Markov property implies that  $\EE_x(\EE_{X_s}( \T_z)\,\1_{ s<\T_z})=\EE_x((\T_z -s) \,\1_{ s<\T_z})$.
%which gives
%$$
%u_{2}(x)=
%.$$
In summary, we get
$$
\EE_x(\T_z^2)=2u_2(x)%=4\int_z^x e^{\gamma(y)} \int_y^\infty u_1(\xi) e^{-\gamma(\xi)}d\xi \,dy
=4\int_z^x e^{\gamma(y)} \int_y^\infty \EE_\xi(\T_z) e^{-\gamma(\xi)}d\xi \,dy.
$$

\noindent {\it Case n}.
Consider $u_n(x)=2(n-1)\int_z^x e^{\gamma(y)} \int_y^\infty u_{n-1}(\xi) e^{-\gamma(\xi)} d\xi \,dy$
which is a $\mathcal{C}^2([0,\infty))$, positive, increasing and bounded solution of
$\,
\LL u_n =-(n-1)u_{n-1} ,\; u_n(z)=0$.
The inductive step assumes that $(n-1)u_{n-1}(x)=\EE_x(\T_z^{n-1})$.
Once again, It\^o's formula (and the boundedness assumptions) implies that 
$\,
u_n(x)=\int_0^\infty \EE_x(\EE_{X_s} (\T_z^{n-1})\1_{s<\T_z} )\,ds=\int_0^\infty\EE_x\big((\T_z-s)^{n-1}\,\1_{s<\T_z}\big)\,ds$.
By binomial expansion, we get %We now use that 
%$\EE_x\big((\T_z-s)^{n-1}\,\1_{s<\T_z}\big)=\sum\limits_{k=0}^{n-1} {n-1 \choose k} \EE_{x}( \T_z^k \,\1_{s<\T_z}) (-s)^{n-1-k}$
%and   get
$$
\begin{array}{ll}
u_n(x)&=\sum\limits_{k=0}^{n-1}  {n-1 \choose k} (-1)^{n-1-k} \, \EE_x\left(\T_z^k  \int_0^{\T_z}  s^{n-1-k} \, ds\right)\\
\\
&=\sum\limits_{k=0}^{n-1}  {n-1 \choose k} \frac{(-1)^{n-1-k}}{n-k} \EE_x(\T_z^n)=
-\frac1n \EE_x(\T_z^n) \sum\limits_{k=0}^{n-1}  {n \choose k} (-1)^{n-k}=\frac1n \EE_x(\T_z^n).
\end{array}
$$
We get for $n\ge 1$
\begin{eqnarray*}
\begin{array}{ll}
\EE_x(\T_z^n)&\hspace{-0.2cm}=nu_n(x)=2n(n-1) \int_z^x e^{\gamma(y)} \int_y^\infty u_{n-1}(\xi) e^{-\gamma(\xi)} d\xi dy\\
&\hspace{-0.2cm}=2n \int_z^x e^{\gamma(y)} \int_y^\infty \EE_\xi(\T_z^{n-1}) e^{-\gamma(\xi)} d\xi dy.
\end{array}
\end{eqnarray*}

\me
(ii) Notice that $\EE_x(\T_z)=u_1(x)< \m(z)$, see \eqref{lyapounov}. Then, we get
that $\EE_x(\T_z^2)\le 2 \m(z)^2$ and inductively we get for $n\ge 1$
$\,
\EE_x(\T_z^n)\le n! \m(z)^n$.
With these bounds we can prove that
$$
\EE_x(e^{\lambda \T_z})=1+\sum\limits_{n=1}^\infty \frac{\lambda^n}{n!}\EE_x(\T_z^n)\le 
1+\sum\limits_{n=1}^\infty (\lambda\m(z))^n\le \frac{1}{1-\lambda\m(z)},
$$
which is finite for all $\lambda< \frac{1}{\m(z)}$. Using the Monotone Convergence Theorem and the previous induction formula, we also conclude that $$
\EE_x(e^{\lambda \T_z})=1+2\lambda\int_z^x e^{\gamma(y)}\int_y^\infty \EE_\xi(e^{\lambda \T_z}) e^{-\gamma(\xi)}d\xi\,dy,$$
which ends the proof.
\end{proof}

\bi{\bf Notation:}  
Some of the computations require to prove a chain of implications about the existence of limits and for that
reason we write $f \rightsquigarrow g$ to mean that 
when $\lim\limits_{z\to \infty} g(z)$ exists then
$\lim\limits_{z\to \infty} f(z)=\lim\limits_{z\to \infty} g(z)$. 

\begin{corollary}
\label{moment2}
Assume $\HH_{1}$ and $\HH_{2}$, then 
\begin{equation}
\label{eq:1}
\lim\limits_{z\to \infty} \frac{\EE_\infty(\T^2_z)}{(\EE_\infty(\T_z))^2}=1.
\end{equation}
\end{corollary}

\begin{proof}
Let us introduce $\ \m_2(z)=\EE_\infty(\T^2_z)$. By  \eqref{for:0} and Markov property, we have
\begin{align*}
\m_2(z)&=4\int_z^\infty e^{\gamma(y)} \int_y^\infty \EE_\xi(\T_z) e^{-\gamma(\xi)} \,d\xi\, dy\\
&=4\int_z^\infty e^{\gamma(y)} \int_y^\infty (\EE_\infty(\T_z)- \EE_\infty(\T_\xi)) e^{-\gamma(\xi)} \,d\xi\, dy\\
&=2\m(z)^2-4\int_z^\infty e^{\gamma(y)} \int_y^\infty \EE_\infty(\T_\xi) e^{-\gamma(\xi)} \,d\xi\, dy.
\end{align*}
by recalling that  $\m(z)=\EE_\infty(\T_z)=
2\int_z^\infty e^{\gamma(y)} \int_y^\infty e^{-\gamma(\xi)} \,d\xi\, dy$. The asymptotic behavior of
$\m_2(z)=\EE_\infty(\T_z^2)$ turns out to be the same as $ \m^2(z)$. For this purpose, it is enough to study the ratio
$$
\frac{4\int_z^\infty e^{\gamma(y)} \int_y^\infty \m(\xi) e^{-\gamma(\xi)} \,d\xi\, dy}{\m^2(z)},
$$
which by l'H\^opital's rule is equivalent to study the ratio $\frac{\int_z^\infty \m(\xi) e^{-\gamma(\xi)} \,d\xi}
{\m(z) \int_z^\infty e^{-\gamma(\xi)} \,d\xi}$. 
We iterate this argument to get
$$
\frac{\int_z^\infty \m(\xi) e^{-\gamma(\xi)} \,d\xi}
{\m(z) \int_z^\infty e^{-\gamma(\xi)} \,d\xi} \rightsquigarrow
\frac{\m(z) e^{-\gamma(z)} }
{\m(z) e^{-\gamma(z)} +2e^{\gamma(z)} \left(\int_z^\infty e^{-\gamma(\xi)} \,d\xi\right)^2}.
$$
Thanks to Corollary \ref{cor:1} (i) (whose proof uses $\HH_{2}$), we have
$$
\frac{e^{2\gamma(z)} \left( \int_z^\infty e^{-\gamma(\xi)}\, d\xi\right)^2}{\m(z)}  
\rightsquigarrow 1-2\q(z) e^{\gamma(z)} \int_z^\infty e^{-\gamma(\xi)}\, d\xi, 
$$
which converges to $0$. Combining these formulas  yields  \eqref{eq:1}.
\end{proof}

\bi
We are now able to prove Proposition \ref{pro:0}.

\begin{proof}[Proof of Proposition \ref{pro:0}]
The equivalences of $(i)$, $(ii)$ and $(iv)$ are proved in \cite[Proposition 7.6]{Cattiaux}. Remark that $(i)\Leftrightarrow (ii)$  means that $\m(0)<\infty$. 
We have shown in Theorem \ref{mom} (\ref{moment-expo}) that if $\lambda<\m(0)$, we have  
$$
\EE_\infty(e^{\lambda \T_0})=\sup_{x\ge 0} \EE_x(e^{\lambda \T_0})\le \frac{1}{1-\lambda\m(0)}<\infty.
$$ Then $(ii)\Rightarrow (iii)$
 and the converse is obvious. 
\end{proof}

\bi
The  next proposition gives a more general recurrence result. 

\noindent 
\begin{proposition}
\label{integ-f} Assume that $f:\RR_+\to \CC$
is a $\mathcal{C}^1$ function such that $|f'(x)|\le A+B e^{\lambda x}$, where $\lambda< \frac{1}{\m(z)}$, then
for all $x\ge z$ we have
\begin{equation}
\label{eq:general}
\hbox{         }\EE_x(f(\T_z))=f(0)+2\int_z^x e^{\gamma(y)} \int_y^\infty \EE_\xi(f'(\T_z)) \, e^{-\gamma(\xi)} \, d\xi dy.
\end{equation}
\end{proposition}

\begin{proof}
According to what we have proved, this relation holds for any polynomial $f$. Then, the result is obtained by approximating
$f'$ by polynomials and passing to the limit, thanks to the growth condition on $f'$ and \eqref{moment-expo}.

\me As a particular case, we have the following relation \begin{equation}
\label{eq:charac}
\EE_x(e^{i\theta \T_z})=1+2i\theta \int_z^\infty e^{\gamma(y)} \int_y^\infty \EE_\xi(e^{i\theta \T_z}) \, e^{-\gamma(\xi)} \, d\xi dy
\end{equation}
 for the characteristic function of $\T_z$.
\end{proof}

\begin{remark}We can reinterpret these formulas using the Green function. Recall that the Green function $g_z:[z,\infty)^2\to \RR_+$, 
for the process $(X_{t\wedge \T_z}: t\ge 0)$ on the domain
$[z,\infty)$, is given by (see \cite{KS}
) $g(x,y)=g_z(x,y)=\int_z^{x\wedge y} e^{\gamma(w)-\gamma(z)}\, dw$  and satisfies for any  measurable function $h$ either positive or 
integrable with respect to $g(x,\xi) e^{-\gamma(\xi)} \, d\xi$ with $h(z)=0$, that
\begin{equation}
\label{eq:green}
\int_0^\infty \EE_x\left(h(X_s)\1_{s<\T_z}\right) ds=\int_z^\infty g(x,\xi) h(\xi)\,  2e^{-(\gamma(\xi)-\gamma(z))} \, d\xi.
\end{equation}

Notice that  (\ref{eq:green}) can be rewritten as 
\begin{equation}
\label{eq:green2}
\EE_x\left(\int_0^{\T_z}  h(X_s) ds\right)%=\int_0^\infty \EE_x\left(h(X_s)\1_{s<\T_z}\right) ds
=2\int_z^x e^{\gamma(y)} \!\!\int_y^\infty h(\xi) e^{-\gamma(\xi)} \, d\xi dy.
\end{equation}
So, if we take $h(\xi)=\ind_{(z,\infty)}(\xi)$, we recover  \eqref{moment}.
%as before $\EE_x(\T_z)=2 \int_z^x e^{\gamma(y)} \int_y^\infty e^{-\gamma(\xi)} \, d\xi dy$.
On the other hand, if we take $h(\xi)=\EE_\xi(f'(\T_z))$ with $f$ as in the statement of Proposition \ref{integ-f},  we obtain
%,  $$
%\begin{array}{l}
%2\int_z^x e^{\gamma(y)} \!\!\int_y^\infty \EE_\xi(f'(\T_z)) e^{-\gamma(\xi)} \, d\xi dy=
%\int_0^\infty \EE_x\left(\EE_{X_s}(f'(\T_z))\1_{s<\T_z})\right) ds=\\
%\int_0^\infty \EE_x\left(f'(\T_z-s)\1_{s<\T_z})\right) ds=
%\EE_x\left(\int_0^{\T_z}  f'(\T_z-s) ds\right)=\EE_x(f(\T_z))-f(0).
%\end{array}
%$$
 Formula (\ref{eq:general}), using the strong Markov property.

%\medskip

%\begin{remark} Another way to write this equality is given by
%$$
%\mathcal{L} \,\EE_x(f(\T_z))=-\EE_x(f'(\T_z)),\, x> z
%$$
%where, as usual, $\LL$ acts on the variable $x$.
\end{remark}

%, proving that
%\begin{equation}
%\label{eq:00}
%\underline\lambda(z)\ge \frac{1}{\m(z)}=\frac{1}{2\int_z^\infty e^{\gamma(y)}\int_y^\infty e^{-\gamma(\xi)}d\xi\,dy}.
%\end{equation}

%%$$
%That is, $U$ is a positive and bounded solution to the equation
%$$
%\LL u=-\lambda u,\, u(z)=1,\, u'(z)=2\lambda e^{\gamma(z)} \int_z^\infty U(\xi,z,\lambda) e^{-\gamma(\xi)}d\xi.
%$$
%Passing to the limit $x\to \infty$ we conclude that (in Section \ref{sec:P_infty} we shall give sense to $\EE_\infty$)
%\begin{equation}
%\label{eq:0}
%\EE_\infty(e^{\lambda \T_z})=
%1+2\lambda\int_z^\infty e^{\gamma(y)}\int_y^\infty \EE_\xi(e^{\lambda \T_z}) 
%e^{-\gamma(\xi)}d\xi\,dy\le \frac{1}{1-\lambda\m(z)}.
%\end{equation}

\section{The limit process $X^\infty$}
\label{sec:P_infty}

First, we introduce the process starting from infinity and check that it satisfies the expected properties.
\begin{theorem}
Assume that $\HH_1$  holds.
 The following limit exists $\PP$-almost-surely for any $t\geq 0$,  
$$
X_t^\infty=\lim\limits_{x\to \infty} X^x_{t\wedge \T^x_0}\in [0,\infty].
$$
The process $X^\infty$ is  continuous, $X_0^\infty=\infty$, $X^\infty_t<\infty$ for all $t>0$ and  
$\T_z^{(\infty)}=\inf\{t\ge 0: X^\infty_t=z\}=\lim\limits_{x\to \infty} \T^{(x)}_z\,$ has some exponential moments
for all $z\ge 0$.

For all $0<s\le t \le \T_0^{(\infty)}$, the process  $X^\infty$ satisfies
\be
\label{EDSlim}
X_t^\infty=X_s^\infty+B_t-B_s-\int_s^t \q(X_u^\infty)\, du.
\ee
\end{theorem}

\begin{proof}  Using that outside a set of $\PP$-measure $0$,   for every $t\geq 0$,  
$X^x_{t\wedge \T^{(x)}_0}$ is increasing in $x$, 
the following limit exists $\PP$-almost-surely for all $t$,  
$$
X_t^\infty=\lim\limits_{x\to \infty} X^x_{t\wedge \T^{(x)}_0}\in [0,\infty].
$$
Moreover $\T^{(x)}_z$ is increasing in $x$ and we set $\T_z=\lim\limits_{x\to \infty} \T^{(x)}_z$. Let us first check that $\T_z= \T^{(\infty)}_z$ for any $z\geq 0$.
Since $\T^{(x)}_{0}\le \T_0$, we have $X^x_{\T_0}=0$ for all $x$ and  then $X^\infty_{\T_0}=0$. Thus, $\T_0^{(\infty)}\leq \T_0$.
On the other hand, if $t<\T_0$, we have
$t<\T^{(x)}_0$ for all large $x\ge x_0=x_0(t,\omega)$, which implies that 
$0<X^{x_0}_t\le X^x_t\le X^\infty_t$ by monotonicity of the flow. 
We conclude  that $\T_0^{(\infty)}=\T_0=\lim\limits_{x\to \infty} \T^{(x)}_z$ and  \eqref{for:0}-\eqref{moment-expo}
ensure by monotone convergence that
$$
\begin{array}{l}
\EE(\T_0^{(\infty)})= \lim\limits_{x\to\infty} \EE(\T_0^{(x)}) =2\int_{0}^\infty e^{\gamma(y)} \int_{y}^\infty e^{-\gamma(\xi)} \, d\xi dy<\infty\\
\EE((\T_0^{(\infty)})^2)= \lim\limits_{x\to\infty} \EE((\T_0^{(x)})^2) =4\int_{0}^\infty e^{\gamma(y)} \int_{y}^\infty \EE(\T_0^{(\xi)}) e^{-\gamma(\xi)} \, d\xi dy<\infty.
\end{array}
$$
Similarly, one can prove that $\T^{(\infty)}_{z}=\lim\limits_{x\to\infty}\T^{(x)}_z$
holds, for all $z$, outside a set of $\PP$-measure 0. Since $\T^{(\infty)}_z\ge \T^{(x)}_z$, we conclude that $\T^{(\infty)}_z$ are positive random variables. 
On the other hand, $z\to \T^{(\infty)}_z$ is non increasing  and  bounded by
the integrable function $\T_0$, which implies
$$
\EE\left(\lim\limits_{z\to \infty} \T^{(\infty)}_z\right)=\lim\limits_{z\to \infty} \EE(\T^{(\infty)}_z)=
\lim\limits_{z\to \infty} 2\int_z^\infty e^{\gamma(y)} \int_y^\infty e^{-\gamma(\xi)}\, d\xi dy=0,
$$
and therefore $\lim\limits_{z\to \infty} \T^{(\infty)}_z=0$ a.s..
\medskip

Now, we shall deduce that  almost surely, $X^\infty_t<\infty$ for all $t>0$. Indeed, for any $0<t_0<t_1$, the event  $\{\sup_{t\in [t_{0}, t_1]} X^x_{t}>A\}$ increases to
$\{\sup_{t\in [t_{0}, t_1]} X^\infty_{t}>A\}$ as $x\rightarrow \infty$ and
$$
\PP\left(\sup_{t\in [t_{0},t_1]} X^\infty_{t}>A \right)= \lim\limits_{x\to \infty} \PP\left( \sup_{t\in [t_{0}, t_1]} X^x_{t} >A\right). 
$$
Let us fix $\varepsilon>0$ and introduce $z = z(\varepsilon)$ such that for any $x\geq z$ large enough $\PP(\T^x_{z} >t_0) <\varepsilon$, using that
 $\lim\limits_{z\to \infty} \T^{(\infty)}_z=0$ and $\T^{(x)}_{z}\leq \T^{(\infty)}_z$ a.s.
 We obtain
$$
\PP\left(\sup_{t\in [t_{0},t_1]} X^\infty_{t}>A \right)\leq \varepsilon + \lim\limits_{x\to \infty} \PP(\T^{(x)}_{z} \leq  t_0, \ \sup_{t\in [t_{0}, t_1]} X^x_{t}>A).
$$
By the strong Markov property at time $\T^{(x)}_{z}$, we obtain for $A>x\geq z(\varepsilon)$,
\begin{align*}
\PP\Big(\T^{(x)}_{z} \leq t_0, \ \sup_{t\in [t_{0}, t_1]} X^x_{t}>A \Big) & \le  \PP \left(\sup_{0\leq t \leq t_1-t_{0}} X_t^z >A\right).
\end{align*}
We then let  $A$ tend to infinity and this term vanishes since $X^z$ does not explode in finite time, which concludes the proof of  $X^\infty_t<\infty$ for all $t>0$.
\medskip

We now prove that $X^{\infty}$ is a continuous process.
We first apply the It\^o's formula to $\m(X^x_{t\wedge \T^{(x)}_0})$ and get
$$
\m(X^x_{t\wedge \T^{(x)}_0})=\m(x)+\int_0^{t\wedge\T^{(x)}_0} \m'(X^x_{s}) \, dB_s+t\wedge \T^{(x)}_0
$$
because $\mathcal{L} \m=1$. As mentioned in the introduction, the functions  $\m$ and $\m'$  are  bounded and continuous functions on $[0,\infty]$. Let us now consider the well defined process on $\RR^+$: 
$$\M_{t}=\int_0^{t\wedge\T_0^{(\infty)}} \m'(X^\infty_{s}) \, dB_s+t\wedge \T_0^{(\infty)}.$$
 Using Doob inequality we get for every $t_0>0$ fixed
\be
\label{vitesse}
&&\EE(\sup_{t\le t_0} |\m(X^{x}_{t\wedge \T^{(x)}_0})-\M_{t}|^2) \nonumber  \\
&& \le  4(\!\m(x))^2 +16\,\EE\left(\int_0^{t_0} \left(\m'(X^{x}_{s\wedge \T_0^{(x)}})-\m'(X^\infty_{s})\right)^2\, ds\right) \\ 
&& +  16\,\EE\left(\int_{t_{0} \wedge \T^{(x)}_0}^{t_0 \wedge \T_0^{(\infty)}} \left(\m'(X^\infty_{s})\right)^2\, ds\right)
+4\EE((\T_0^{(\infty)}-\T^{(x)}_0)^2). \nonumber
\ee
When $x$ tends to infinity, the first term of the RHS  tends to $0$ by $\HH_{1}$. The second term tends to $0$  
by dominated convergence theorem (since $\m'$ is bounded). The third term is bounded (up to a constant) by $\EE(\T_0^{(\infty)}-\T^{(x)}_0)$. It tends to $0$ as the forth term by dominated convergence theorem. We have thus proved that almost-surely, for any $t$,
 $$\M_{t} = \lim\limits_{x\to \infty} \m(X^{x}_{t\wedge \T^{(x)}_0}).$$ 
By continuity of $\m$, we deduce
$$
\m(X^\infty_t)=\int_0^{t\wedge\T_0^{(\infty)}} \m'(X^\infty_s) \, dB_s+t\wedge \T_0^{(\infty)}.
$$
This implies that the process $t \to \m(X^\infty_{t})$ is continuous and since $\m$ is a  diffeomorphism, we conclude that the process $t\to X^\infty_{t}$ is also continuous.  

 \me We conclude then easily that for all $0<s\le t \le \T_0^{(\infty)}$, the process  $X^\infty$ satisfies
 the SDE \eqref{EDSlim} for any $t>s>0$, by letting $x\rightarrow \infty$ and using $\m'$ bounded. 
\end{proof}

We obtain now the convergence of processes when the initial condition goes to infinity and provide the speed of this convergence in terms of the Lyapunov function
$\m$. For that purpose, we introduce the distance  $d_{\m}(x,y)=|\m(x)-\m(y)|$ and endow  $\RR_+\cup\{\infty\}$ with this distance which makes it a Polish space.

\begin{theorem} \label{cvetvit} Under   $\HH_1$,  
for any $T>0$, the law   of $(X^x_{t\wedge\T^{(x)}_0} : t\in [0,T])$ converges to the  law of $(X_t^\infty : t\in [0,T])$ as $x\to\infty$ in the space of probabilities on $C([0,T], \mathbb{R}_{+}\cup \{\infty\})$.  \\
%Moreover, there exists $C> 0$ such that for any $T>0$,
%$$\EE\left(\sup_{t\leq T_0^{x} \wedge T} d_{\m}(X_t^{\infty},X_s^x)^2 \right)\leq d_{\m}(x,\infty)  e^{CT}.$$
Assuming further  
$$\lim\limits_{x\to\infty} q(x)=\infty \  \text{  and } \  \limsup\limits_{x\to\infty} \frac{|q'(x)|}{q(x)}<\infty,$$
 there exist  $c_1,c_2>0$ such that for any $x\geq 0$ and $T>0$,
$$\EE\left(\sup_{t\leq  T\wedge \T^{(x)}_0} d_{\m}(X_{t}^{\infty},X_{t}^x)^2 \right)\leq c_1 \m^2(x)  e^{c_2T}.$$
\end{theorem}
Thus, $\m(x)=d_{\m}(x,\infty)$ provides the speed of the convergence of the flow when $x\rightarrow \infty$. Letting $T=0$ in this inequality shows that it is sharp.
\begin{proof}
We recall that $\M_{t}=\m(X_t^{\infty})$ where $X_t^{\infty}=X_{t\wedge \T_0^{(\infty)}}^{\infty}$. Then,  \eqref{vitesse} is given by
\ben
&&\EE\left(\sup_{t\le T} d_{\m}\left(X^{x}_{t\wedge \T^{(x)}_0},X_{t}^{\infty} \right)^2\right)  \\
&& \leq  4\!\m^2(x) +16\,\EE\left(\int_0^{T} \left(\m'(X^{x}_{s\wedge \T_0^{(x)}})-\m'(X^\infty_{s})\right)^2\, ds\right) \nonumber\\ 
&& +  16\, \parallel \m'\parallel_{\infty} \EE\left( \T_0^{(\infty)}-\T^{(x)}_0\right)+4\EE((\T_0^{(\infty)}-\T^{(x)}_0)^2),
\een
where each term of the right hand side goes to $0$. This ends up the proof of the first part of the 
statement. Similarly, as in  \eqref{vitesse} we have (using that $\T^{(x)}_{0}\leq \T^{(\infty)}_{0}$),
\ben
&&\EE\left(\sup_{t\le T\wedge \T^{(x)}_0} d_{\m}\left(X^{x}_{t},X_{t}^{\infty} \right)^2\right)=
\EE\left(\sup_{t\le T\wedge \T^{(x)}_0} |\m(X^{x}_t)-\M_{t}|^2\right)  \\
&&\le  4\m^2(x) +16\,\EE\left(\int_0^{T\wedge \T_0^{(x)} } \left(\m'(X^{x}_s)-\m'(X^\infty_{s})\right)^2\, ds\right) 
\een
Moreover, our assumptions here ensure that $\m''/\m'$ is bounded, see Lemma \ref{lip} for the proof. Then there there exists $c_1 >0$ such that for any $y\geq x \geq 0$,
$$\vert \m'(x)-\m'(y) \vert \leq \int_x^y \vert \m''(z)\vert ds  \leq c_1 \int_x^y \vert \m'(z) \vert dz =c_1 d_{\m}(x,y), $$
since $\m'$ has a constant sign. The second part of the statement then follows from  Gronwall inequality.
\end{proof}

%\marginpar{ FAIRE LE LIEN AVEC NOTRE PAPIER AVEC VINCENT ET MATHIEU}

\section{ Proofs of  Theorem \ref{the:1} and Corollary \ref{cor:Pierre}}

Let us assume in this section  that $ \HH_1$ and  $\HH_2$ hold. 
We have seen in the previous section that
$\lim\limits_{z\to \infty} \T^{(\infty)}_z = 0 \  a.s.$  and $\lim\limits_{z\to \infty}  \EE(\T^{(\infty)}_z)=0$.

\begin{proof}[Proof of Theorem \ref{the:1}]
The first statement  follows immediately from the behavior  of the second moment. 
Indeed let  $\epsilon>0$ and use Markov inequality to get
$$
\PP\left(\left|\frac{\T^{(\infty)}_z}{\EE(\T^{(\infty)}_z)}-1\right|\ge \epsilon\right)
\le \frac{1}{\epsilon^2} \left(\frac{\EE((\T^{(\infty)}_z)^2)}{(\EE(\T^{(\infty)}_z))^2}-1\right).
$$
This right hand side converges to zero from \eqref{eq:1}, and the result is proved.

\medskip

The second statement  is more involved.
Recall that the function  $\m(w)=\EE(\T^{(\infty)}_w)$ is positive and decreasing. Hence, for $0<\rho<1$, 
we can define recursively $z_n$ for  $n\ge 0$ by $z_{0}>0$ and 
$$
z_n=\inf\{ w>z_{n-1}:\,  \m(w)=  (1-\rho) \m(z_{n-1})\}.
$$
It is clear that $(z_n: n\ge 0)$ is a strictly increasing sequence. Let us prove that $z_n\to \infty$. Indeed, we have
for $n\ge 1$,
$\m({z_n})= (1-\rho)^{n} \m(z_0),$
which tends to $0$ as $n$ tends to infinity,
implying that $z_n\to \infty$.

\smallskip We first prove that  
\begin{equation}
\label{LLN}\lim\limits_{n\to \infty} \frac{\T^{(\infty)}_{z_n}}{\EE(\T^{(\infty)}_{z_n})}=1
\end{equation}
holds a.s. using  the decomposition
$$\T^{(\infty)}_{z_n} = \sum\limits_{k\ge n+1}\tau_k,  \qquad  \text{with } \  \tau_k=\T^{(\infty)}_{z_{k-1}}-\T^{(\infty)}_{z_k}.$$

\noindent Since the process $X^{\infty}$ satisfies the strong 
Markov property, the random variables $(\tau_k)_{k\geq 0}$  are independent with the same 
distribution as $\T_{z_{k-1}}^{(z_k)}$. Recalling that $\EE(\T_z^{(\infty)}) = \m(z),$
%Indeed, for continuous bounded functions $f$ and $g$ and $k\leq p-1$, we get  by Markov property that
%\begin{align*}
%\EE(f(\T^{(\infty)}_{z_{p-1}}-\T^{(\infty)}_{z_p}) g(\T^{(\infty)}_{z_{k-1}}-\T^{(\infty)}_{z_k})) &= \lim\limits_{m\to \infty} \EE_{z_{m}}(f(\T_{z_{p-1}}-\T_{z_p}) g(\T_{z_{k-1}}-\T_{z_k})) \\
%&= \EE(f(\T^{(\infty)}_{z_{p-1}}-\T^{(\infty)}_{z_p})) \EE_{z_{k}}(g(\T_{k-1})).
%\end{align*}
% We denote by $(\tau_k, k\ge 1)$ a collection of independent random variables such that $\tau_k$ has the same distribution as $\T^{(\infty)}_{z_{k-1}}-\T^{(\infty)}_{z_k}$.
we shall prove that
$$
\lim\limits_{n\to \infty} \frac{1}{\m(z_{n})}  \sum\limits_{k\ge n+1} (\tau_k-\EE(\tau_k))=0\,\hbox{ a.s.}
$$
to get  \eqref{LLN}.
We  will first prove that  
$$
\sum\limits_{k\ge 1} \frac{\hbox{Var}(\tau_k)}{\m^2(z_{k-1})}<\infty,
$$
and using that $\m(z_{k})$ is decreasing, we apply  Proposition 1 in Klesov \cite{klesov}.
Note that since the $(\tau_{k})_{k\geq 0}$ are independent,  
$$
\frac{\hbox{Var}(\tau_k)}{\m^2(z_{k-1})}= \frac{\hbox{Var}(\T^{(\infty)}_{z_{k-1}})-\hbox{Var}(\T^{(\infty)}_{z_{k}})}{\m^2(z_{k-1})}.
$$
Since by \eqref{eq:2}, $V(w)=\hbox{Var}(\T^{(\infty)}_w)=8\int_w^\infty e^{\gamma(y)} \int_y^\infty e^{\gamma(\xi)} 
\left(\int_{\xi}^\infty e^{-\gamma(\eta)} \,d\eta\right)^2\, d\xi\, dy$ is a decreasing function, we get
$$
\sum\limits_{k\ge 1} \frac{\hbox{Var}(\tau_k)}{\m^2(z_{k-1})}=\sum\limits_{k\ge 1} \frac{\int_{z_{k-1}}^{z_k} -V'(w)\, dw}{\m^2(z_{k-1})}
\le \sum\limits_{k\ge 1} \int_{z_{k-1}}^{z_k} \frac{-V'(w)}{\m^2(w)}\, dw=\int_{z}^\infty \frac{-V'(w)}{\m^2(w)}\, dw.
$$
This last integral is finite (or infinite) if it is finite (or infinite) for an equivalent function (see Corollary \ref{cor:1} 
(i),(ii) and (iii))
$$
\frac{-V'(w)}{\m^2(w)}=8\frac{e^{\gamma(w)} \int_w^\infty e^{\gamma(\xi)} 
\left(\int_{\xi}^\infty e^{-\gamma(\eta)} \,d\eta\right)^2\, d\xi}{\m^2(w)} \approx  \frac1{\q^3(w) (\int_w^\infty \q^{-1}(x))^2 dx}.
$$
Our extra hypothesis ensures  that this last function is integrable. 
This yields the a.s. convergence for the desired ratio along the subsequence $(z_n)_n$.

\smallskip

Now, we prove that this sequence dominates the full path.  Indeed, we consider $z>z_{0}$ and $n\ge 1$ such that $z\in [z_{n-1},z_n[$. Then by the obvious monotonicity 
of $\T^{(\infty)}_z$ and $ \m(z)$ on $z$, we get
$$
\frac{\T^{(\infty)}_{z_{n-1}}}{\m(z_{n-1})}\ge \frac{\T^{(\infty)}_z}{\m(z_{n-1})}
\ge (1-\rho)\frac{\T^{(\infty)}_z}{\m(z)}\\
\ge (1-\rho)\frac{\T^{(\infty)}_{z_n}}{\m(z_{n-1})}=(1-\rho)^2\frac{\T^{(\infty)}_{z_n}}{\m(z_n)}.
$$
Therefore almost surely,
$$
1-\rho \le \liminf\limits_{z\to \infty} \frac{\T^{(\infty)}_z}{\m(z)} \le \limsup\limits_{z\to \infty} \frac{\T^{(\infty)}_z}{\m(z)} \le \frac{1}{1-\rho} .
$$
Taking  $\rho\to 0$ ends the proof of   Theorem \ref{the:1}.
\end{proof}

\bi

\begin{proof}[Proof of Corollary \ref{cor:Pierre}] Inequality (\ref{moment-expo}) shows that
for all $\lambda<1$, we have
$$
\EE\left(\exp\left(\lambda\, \frac{\T^{(\infty)}_z}{\EE(\T^{(\infty)}_z)}\right)\right)\le \frac{1}{1-\lambda}.
$$
This means that $\left(\exp\left(\lambda\, \T^{(\infty)}_z/\EE(\T^{(\infty)}_z)\right), z\ge 1\right)$ is a tight family (consider $\lambda<\lambda'<1$) and the result
follows from the convergence in probability of $ \frac{\T^{(\infty)}_z}{\EE(\T^{(\infty)}_z)}$ to 1. 
Finally, $(ii)$ is shown similarly.
\end{proof}

\bigskip 

\section{Proof of  Theorem \ref{the:2}}
We will use L\'evy's Theorem and prove that for all $t$ small enough and all $z$ large enough,
$$
\EE\left(\exp\left(it
\frac{\T^{(\infty)}_z-\EE(\T^{(\infty)}_z)}{\hbox{Var}(\T^{(\infty)}_z)^{1/2}}\right)\right)= 
\exp\left(-\frac{t^2}{2}\right)\left(1+o(1)\right).
$$

As for Theorem \ref{the:1}, the proof relies on the fact that $\T^{(\infty)}_z$ can be seen as a 
sum of independent random variables, which are adapted to estimate the variance. For that
purpose, we consider $\rho\in (0,1)$ and define recursively another sequence $z_n$ by $z_0=z$ and for $n\ge 1$
$$
z_{n}=\inf\{x> z_{n-1}: \hbox{Var}(\T^{(x)}_{z_{n-1}})=\rho\,  \hbox{Var}(\T^{(\infty)}_{z_{n-1}})\}.
$$
Note that the function $x\to  \hbox{Var}(\T^{(x)}_{z_{n-1}})$ is  continuous, vanishing at 
 $z_{n-1}$ and positive at infinity. The sequence $(z_n:n\ge 0)$ is well
defined and it is strictly increasing. We denote by $\bar z\in (z,\infty]$ its limit and first  prove that $\bar z = \infty$.  \\
As in the proof of Theorem \ref{the:1},  $\tau_n=\T^{(\infty)}_{z_{n-1}}- \T^{(\infty)}_{z_{n}}$ are  independent  and
%We deduce $$\hbox{Var}_\infty(\T_z)=\hbox{Var}_\infty(\T_{z_{n}})+\sum \limits_{k=1}^n \hbox{Var}(\tau_k).$$
$$ 
\hbox{Var}(\tau_n)=\rho\,\hbox{Var}(\T^{(\infty)}_{z_{n-1}}),
$$
by construction, since $\tau_n$ is distributed as $\T^{z_n}_{z_{n-1}}$.
This implies $\hbox{Var}(\T^{(\infty)}_{z_{n-1}})=\rho \hbox{Var}(\T^{(\infty)}_{z_{n-1}})+\hbox{Var}(\T^{(\infty)}_{z_{n}})$ and by induction
$$
\begin{array}{l}
\hbox{Var}(\T^{(\infty)}_{z_{n}}) = (1-\rho)^n\,\hbox{Var}(\T^{(\infty)}_z).
\end{array}
$$
Letting $n\rightarrow \infty$ and using that
  $\hbox{Var}(\T^{(\infty)}_{z_{n}})\rightarrow \hbox{Var}(\T^{(\infty)}_{\bar z})$ from 
(\ref{eq:2}), we get
 $$\lim\limits_{n\rightarrow \infty} z_n=\bar z=\infty.$$
Moreover, writing $\sigma(z) = \hbox{Var}^{1/2}(\T^{(\infty)}_z)$, we have 
\be
\label{calvar}
\hbox{Var}(\tau_n)=\rho(1-\rho)^{n-1}\sigma(z)^2.
\ee
\medskip

Now we turn to the proof of the convergence of characteristic function. Recalling that
%So far the conclusion is that $\bar z=\infty$ and that 
$\T^{(\infty)}_z - \EE(\T^{(\infty)}_{z})= \sum\limits_{k\ge 1} (\tau_k-\EE(\tau_{k}))$,
%\medskip
%Let us denote .
%Since  $\,\hbox{Var}\Big( \frac{\tau_{k} - \EE(\tau_{k})}{\sigma_{z}}\Big) = \rho(1-\rho)^{k-1}$,
%the series on the right hand side converges $a.s.$ by Kolmogorov Theorem. 
%Hence, 
we have for a fixed $t\in \mathbb{R}$, 
$$
\EE\left(e^{it \ \frac{\T^{(\infty)}_z - \EE(\T^{(\infty)}_{z})}{\sigma_{z}}}\right) = 
\prod_{k\ge 1} \EE\left(e^{it \ \frac{\tau_{k} - \EE(\tau_{k})}{\sigma_{z}}}\right).
$$ 
Using  \eqref{calvar} and  estimate (27-11) in \cite{billingsley}, we get
\begin{align*}
\left| \EE\left(e^{it \ \frac{\tau_{k} - \EE(\tau_{k})}{\sigma_{z}}}\right) - e^{-{t^2\over 2} \rho(1-\rho)^{k-1}}\right| 
\le {|t|^3\over \sigma(z)^3}\EE\left(|\tau_{k} - \EE(\tau_{k})|^3\right) + C{t^4}  \rho^2(1-\rho)^{2k},
\end{align*}
where $C$ is a positive constant such that $|e^{-x}-1 +x| \le C x^2$ for any $x\ge 0$.
Then,
\begin{align*}
\left|\EE\left(e^{it \ \frac{\T^{(\infty)}_z - \EE(\T^{(\infty)}_{z})}{\sigma_{z}}}\right)- e^{-{t^2\over 2}}\right| &= 
\left|\prod_{k\ge 1}\EE\left(e^{it \ \frac{\tau_{k} - \EE(\tau_{k})}{\sigma_{z}}}\right) -
\prod_{k\ge 1} e^{-{t^2\over 2} \rho(1-\rho)^{k-1} }\right|\\
&\le {|t|^3\over \sigma_{z}^3}\ \sum_{k\ge 1} \EE\left(|\tau_{k} - \EE(\tau_{k})|^3\right) + C{t^4} {\rho\over 2 - \rho}.
\end{align*}
It remains to estimate $ {1\over \sigma_{z}^3}\ \sum_{k\ge 1} \EE\left(|\tau_{k} - \EE(\tau_{k})|^3\right)$ and let $\rho$ go to $0$ to check that the right hand side vanishes.
 We notice that
using twice Cauchy-Schwartz inequality
\ben
&&{1\over \sigma(z)^3}\ \sum_{k\ge 1} \EE\left(|\tau_{k} - \EE(\tau_{k})|^3\right)\\
&& \qquad \le 
\sum_{k\ge 1} \left({1\over \sigma_{z}^4} \EE\left(|\tau_{k} - \EE(\tau_{k})|^4\right)\right)^{1/2}
\left({1\over \sigma(z)^2}\EE\left(|\tau_{k} - \EE(\tau_{k})|^2\right)\right)^{1/2}\\
&&\qquad
\le \left({1\over \sigma(z)^4} \sum_{k\ge 1} \EE\left(|\tau_{k} - \EE(\tau_{k})|^4\right)\right)^{1/2}
\left({1\over \sigma(z)^2} \sum_{k\ge 1} \EE\left(|\tau_{k} - \EE(\tau_{k})|^2\right)\right)^{1/2}\\
&&\qquad = \left({1\over \sigma(z)^4} \sum_{k\ge 1} \EE\left(|\tau_{k} - \EE(\tau_{k})|^4\right)\right)^{1/2},
\een
where  the last equality comes from $\sigma(z)^2=\hbox{Var}(\T^{(\infty)}_z)=\sum_{k \geq 1} \hbox{Var}(\tau_k)$ or \eqref{calvar}.
Theorem \ref{the:2}  follows now from the next lemma by letting $\rho$ tend to $0$.
\endproof

\medskip
\begin{lemma}We have for any $\rho\in (0,1)$,
$$
\lim\limits_{z\to \infty} \frac{1}{\sigma_{z}^4} \sum\limits_{k\ge 1} \EE\left([\tau_k-\EE(\tau_k)]^4\right)  =\frac{3\rho}{2-\rho}.
$$
\end{lemma}

\begin{proof}
Since  the  random variables $(\tau_k -\EE(\tau_k))$ are independent and centered, and since
$\T_z^{(\infty)}=\sum_{k\geq 1} \tau_k$ has finite moments, we can write 
$$
\EE\left([\T^{(\infty)}_z-\EE(\T^{(\infty)}_z)]^4\right)=\sum\limits_{k\ge 1}
\EE\left([\tau_k-\EE(\tau_k)]^4\right) +  
3 \sum\limits_{k\neq j}
\EE\left([\tau_k-\EE(\tau_k)]^2\right)\EE\left([\tau_j-\EE(\tau_j)]^2\right).
$$
Then, using \eqref{calvar}, 
\begin{align*}
&\EE\left([\T^{(\infty)}_z-\EE(\T^{(\infty)}_z)]^4\right)\\
&=\sum\limits_{k\ge 1} \EE\left([\tau_k-\EE(\tau_k)]^4\right)+3\sum\limits_{k, j} \hbox{Var}(\tau_k)\hbox{Var}(\tau_j)
-3\sum\limits_{k\ge 1}\left( \hbox{Var}(\tau_k)\right)^2\\
&=\sum\limits_{k\ge 1}
\EE\left([\tau_k-\EE(\tau_k)]^4\right)+3\left(\hbox{Var}(\T^{(\infty)}_z)\right)^2 
-3\rho^2\left(\hbox{Var}(\T^{(\infty)}_z)\right)^2 \sum\limits_{k\ge 1}
(1-\rho)^{2(k-1)}\\ 
&=\sum\limits_{k\ge
  1}\EE\left([\tau_k-\EE(\tau_k)]^4\right)+3\left(\hbox{Var}(\T^{(\infty)}_z)\right)^2 
-3\left(\hbox{Var}(\T^{(\infty)}_z)\right)^2 \frac{\rho^2}{1-(1-\rho)^2}\\
&=\sum\limits_{k\ge 1}
\EE\left([\tau_k-\EE(\tau_k)]^4\right)
+3\left(\hbox{Var}(\T^{(\infty)}_z)\right)^2\left(
1- \frac{\rho}{2-\rho}\right). 
\end{align*}
It is proved in  Lemma \ref{eq:fourth} that 
$$\lim\limits_{z \to \infty}
\frac{\EE\left([\T^{(\infty)}_z-\EE(\T^{(\infty)}_z)]^4\right)}
{(\hbox{Var}(\T^{(\infty)}_z))^2}=3 
$$
and we conclude by dividing the last identity by 
$\sigma^4_z=(\hbox{Var}(\T^{(\infty)}_z))^2$.
\end{proof}

\section{Proofs of Theorem \ref{fluctuX} and Remark \ref{rem:remhypop}}

We first derive some consequences of the assumptions of Theorem  \ref{fluctuX} and recall from \eqref{eq:Mgrande} the notation
$
M(z)=\int_{z}^{\infty}\frac{1}{q(x)}\;dx.
$
%$\HH_{1}$,  $\lim\limits_{x\to\infty}q(x)=\infty$,  and   \eqref{eq:H2.5.2}.

\begin{lemma}
\label{conhypop}
Assume $\HH_{1}$, $\lim\limits_{x\to\infty}q(x)=\infty$   and \eqref{eq:H2.5.2}. Then
\begin{enumerate}[(i)]
\item
$
\lim\limits_{z\to\infty} q(z)\sqrt{M(z)}=\infty.
$
\item $\HH_{2}$ holds.
\item $\lim\limits_{z\to \infty} \frac{\m'(z)}{\sqrt{\m(z)}}=0$.
\end{enumerate}
\end{lemma}
\begin{proof}
Condition \eqref{eq:H2.5.2} can be reformulated as
\be
\label{reform}
\lim\limits_{z\to\infty} \frac{M''(z)}{M'(z)}\;\sqrt{M(z)}=0.
\ee
In other words, for any $\epsilon>0$ there exists $z(\epsilon)>0$ such
that for any $z>z(\epsilon)$ we have $M'(z)<0$ and
%$\epsilon \;\frac{M'(z)}{\sqrt{M(z)}}<
$M''(z)<-\epsilon M'(z)/\sqrt{M(z)}$.
Integrating between $z>z(\epsilon)$ and infinity we get
$0\le -M'(z)\le 2\;\epsilon \sqrt{M(z)}$.
This implies
\begin{equation}\label{limM}
\lim\limits_{z\to\infty} \frac{M'(z)}{\sqrt{M(z)}}=0\;,
\end{equation}
which is assertion (\textit{i}).
 
We observe that
$$
\frac{q'(z)}{q(z)^{2}}=\left(\frac{M''(z)}{M'(z)}\;\sqrt{M(z)}\right)\;
\left(-\frac{M'(z)}{\sqrt{M(z)}}\right)
$$
and therefore $\HH_{2}$ holds and (\textit{ii}) is proved. Furthermore, using Corollary \ref{cor:1} $(i)$,$(ii)$ and \eqref{limM}
we deduce $\lim\limits_{z\to \infty} \frac{\m'(z)}{\sqrt{\m(z)}}=0$ and (\textit{iii}) is proven.
\end{proof}

We can now prove Theorem \ref{fluctuX}. According to Lemma \ref{conhypop}, see Remark \ref{rem:remhypop}, 
we can assume that $\HH_2$ holds. For $y\in\RR$ fixed, we define the function
$$
r_y(t)=\m^{-1}(t)+y\,\sqrt t \hbox{ for } t\in\RR_{+}.
$$
Recall that $\m(z)=\EE(\T_z^{(\infty)})$ and $\sigma(z) = \hbox{Var}^{1/2}(\T^{(\infty)}_z)$.

First, we observe that 
\be
\label{inclus}
\left\{\frac{X_t^{\infty}-\m^{-1}(t)}{\sqrt t}<y\right\} \subset \{\T_{r_y(t)}<t\}
\ee
and the derivation of the upperbound of the CLT comes from the CLT for $\T$. Indeed
 $r_{y}(t)\rightarrow \infty$ as $t\rightarrow 0$ since by 
$\HH_{1}$,   $\m(z)\rightarrow 0$ as $z\rightarrow \infty$
and by Lemma \ref{tec} $(i)$,
\be
\lim\limits_{t\searrow 0}\, \frac{t- \m(r_{y}(t))}{\sigma(r_{y}(t))}=\sqrt{\Sigma}\, y.
\ee
Then Theorem \ref{the:2} ensures that
\be
\lim\limits_{t\searrow 0} \PP\big(\T^{(\infty)}_{r_{y}(t)}<t\big)&=&\lim\limits_{t\searrow 0} \PP\bigg(\frac{\T^{(\infty)}_{r_{y}(t)}-\m(r_y(t))}{\sigma(r_y(t))}<\frac{t- \m(r_{y}(t))}{\sigma(r_{y}(t))}\bigg)\nonumber \\
&=&\Phi(\sqrt{\Sigma}\, y), \label{eq:limPHI}
\ee
where 
  $\Phi$ is the cumulative distribution function of   the standard normal distribution.  From \eqref{inclus}, we then obtain
$$
\limsup\limits_{t\searrow 0} \PP\left(\frac{X^{\infty}_t-\m^{-1}(t)}{\sqrt
  t}<y\right)\,\le \Phi\big(\sqrt{\Sigma}\, y\big)
$$
and we now prove a reverse estimate. Let $\epsilon>0$, we have
 \be
 \label{union}
 \big\{\T^{(\infty)}_{r_{y}(t)}\le t\big\} \subset A_t
\bigcup \big\{
X^{\infty}_{t}\le r_{y,\epsilon}(t)
\big\},
\ee
where $r_{y,\epsilon}(t)=r_{y}(t)+\epsilon\sqrt{t}$ and  the event $A_t$ is defined by $$A_t=\big\{
\T^{(\infty)}_{r_{y}(t)}\le t\,; \, \exists\, u\in\big[\T^{(\infty)}_{r_{y}(t)},t\big], \ 
X^{\infty}_{u}>r_{y,\epsilon}(t).
\big\}$$
By the strong Markov property we have
$$
\PP(A_t) \leq \PP\big(\T^{(\infty)}_{r_{y}(t)}\le t\big)\; 
\PP_{r_{y}(t)}\big(\exists\,
s\in\big[0,t\big], \  
X_{s}>r_{y,\epsilon}(t)
\big) 
$$
and using  $\PP_{r_{y}(t)}\big(\exists\,
s\in\big[0,t\big], \  
X_{s}>r_{y,\epsilon}(t)
\big) \leq \PP_{r_{y}(t)}\big(\T_{r_{y,\epsilon}(t)}\le \T_{0}
\big)$ since $0$ is absorbing, we get from \eqref{union}
$$\PP\big(X_{t}^{\infty}\le r_{y,\epsilon}(t)\big) \geq 
\PP\big(\T_{r_{y}(t)}^{(\infty)}\le t\big)\big(1-\PP_{r_{y}(t)}\big(\T_{r_{y,\epsilon}(t)}\le \T_{0}
\big)\big).
$$
As $\{X^{\infty}_t<r_{y,\epsilon}(t)\}=\left\{(X^{\infty}_t-\m^{-1}(t))/\sqrt{
  t}<y+\epsilon\right\}$, we now need to check that for any $\epsilon>0$,
\be
\label{conclu}
\lim\limits_{t\searrow0}
\PP_{r_{y}(t)}\big(\T_{r_{y,\epsilon}(t)}\le \T_{0}\big)=0
\ee
and using  \eqref{eq:limPHI}, we will get 
$$
\liminf\limits_{t\searrow 0} \PP\left(\frac{X^{\infty}_t-\m^{-1}(t)}{\sqrt
  t}<y+\epsilon\right)\geq \Phi\big(\sqrt{\Sigma}\, y\big)
$$
and complete  the proof of 
 Theorem \ref{fluctuX} (replace $y$ by $y-\epsilon$). \\
To prove \eqref{conclu}, we recall  that the scale function
$\Lambda(x)=\int_{0}^{x}e^{\gamma(\xi)}d\xi$ satisfies 
$\LL \Lambda=0, \Lambda(0)=0$. By It\^o's formula, we have for any $x>z>0$,
$\,
\EE_{z}\big(\Lambda\big(\T_{x}\wedge\T_{0}\bi\big)\big)=\Lambda(z)\;. 
$
Therefore
$\,
\PP_{z}\big(\T_{x}<\T_{0}\big)=\frac{\Lambda(z)}{\Lambda(x)}\;
\,$
and
$$
\PP_{r_{y}(t)}\big(\T_{r_{y,\epsilon}(t)}\le T_{0}
\big)=\frac{\Lambda\big(r_{y}(t)\big)}{\Lambda\big(r_{y}(t)+\epsilon\,\sqrt{t}\big)}.
$$
Combing  Lemma \ref{tec}  $(ii-iv)$, the left hand side goes to $0$, which ends the proof of Theorem \ref{fluctuX}. \\

We finish this section with a proof of the fact given in Remark \ref{rem:remhypop} about the hypotheses of Theorem \ref{fluctuX}.

\begin{proposition} Assume that $\HH_1$ and $\HH_2$ hold. We also assume that
\begin{equation}
\label{eq:new_b}
b=\lim\limits_{z\to \infty} q'(z)\int_z^\infty \frac{1}{q(x)}\, dx\in \RR.
\end{equation}
Then, $b\ge 0$ and \eqref{eq:H2.5.1} and  \eqref{eq:H2.5.2} hold with $\Sigma=2b+1$.
\end{proposition}

\begin{proof} Notice first that since $q(z)\to \infty$, when $z\to \infty$, we conclude there exists a sequence
$z_n\to \infty$ for which $q'(z_n)>0$, which together with the existence of the limit in \eqref{eq:new_b}
imply that $b\ge 0$.

Let us prove \eqref{eq:H2.5.2}, that is,
$\lim\limits_{z\to\infty}\frac{q'(z)}{q(z)}\sqrt{\int_{z}^{\infty}\frac{1}{q(x)}\;dx}=0$.
Notice that 
$$
\left(\frac{q'(z)}{q(z)}\right)^2 \int_{z}^{\infty}\frac{1}{q(x)}\;dx=\frac{q'(z)}{q^2(z)} \, q'(z)\int_{z}^{\infty}\frac{1}{q(x)}\;dx,
$$
which converges to $0$ by $\HH_2$ and \eqref{eq:new_b}.

\medskip
Let us now show \eqref{eq:H2.5.1}. By l'H\^opital's rule, we get that 
$$
\begin{array}{l}
\Sigma=\lim\limits_{z\to\infty}\frac{\int_{z}^{\infty}q(x)^{-1}\;dx}{q^{2}(z)\;\int_{z}^{\infty}q(x)^{-3}dx}=
\lim\limits_{z\to\infty} \frac{q^{-2}(z)\int_{z}^{\infty}q(x)^{-1}\;dx}{\int_{z}^{\infty}q(x)^{-3}dx}\\
\\
=\lim\limits_{z\to\infty} \frac{-2q^{-3}(z)q'(z)\int_{z}^{\infty}q(x)^{-1}\;dx-q^{-3}(z)}{-q^{-3}(z)}
=\lim\limits_{z\to\infty} 2\,q'(z)\int_{z}^{\infty}q(x)^{-1}\;dx+1=2b+1.
\end{array}
$$
The result is shown.
\end{proof}

\section{Proof of Theorem \ref{the:3}}
 We refer to Appendix C and to \cite{Cattiaux} for all  results recalled in this section and the next one. 

\me We assume that $\HH_1$ and $\HH_2$ hold.
We begin by noticing that 
$$
\lambda_{1}(z)> \frac{1}{2\int_z^\infty
  e^{\gamma(y)}\int_y^\infty e^{-\gamma(\xi)}d\xi\,dy}. 
$$
Indeed, let $\nu=\nu_z$ be the unique q.s.d. on $[z,\infty)$. This q.s.d. has a density proportional to
$\psi_{z,1}$ with respect to $\mu_z$.
We know that $\T_z$ is exponentially distributed if the initial law is $\nu$, and the mean of $\T_z$ is $1/\lambda_{1}(z)$, that is
$$
\lambda_{1}(z)=\frac{1}{\EE_\nu(\T_z)}>\frac{1}{\EE_\infty(\T_z)}=
\frac{1}{2\int_z^\infty e^{\gamma(y)}\int_y^\infty e^{-\gamma(\xi)}d\xi\,dy}.
$$

We shall prove that the tail of the distribution of $\T_z$, under $\PP_\infty$, has an exponential tail with decay rate $\lambda_{1}(z)$.
This is a consequence of Theorem \ref{the:3} that we prove now.

\me 
An important result for our proof is the following lemma. 

\begin{lemma} 
\label{lem:0}
For every $z\ge 0$, the process $(e^{\lambda_{1}(z) (t\wedge
  \T_z)} \psi_{z,1}(X_{t\wedge \T_z}):\, t\ge 0)$  
is a nonnegative  martingale and for any
$z<x<w\le \infty$
$$
\psi_{z,1}(x)\EE_w\left(e^{\lambda_{1}(z)\T_x}\right)=\psi_{z,1}(w).
$$
\end{lemma}
\begin{proof}  Recall that ${\cal L} \psi_{z,1} = - \lambda_{1}(z)\psi_{z,1}$ and then the process is a local martingale for any initial condition.  In Appendix Theorem C.1 and Proposition C.2 (ii), it is proved that  $\psi_{z,1}$ is positive and bounded. 
We recall that $\lambda_{1}(\bullet)$ is strictly increasing, so $\lambda_{1}(z)<\lambda_{1}(x)$. On the other
hand under $\HH_1, \HH_2$  for any 
$\lambda'<\lambda_{1}(x)$  (see \cite{Cattiaux} Corollary 7.9)
$$
\sup_{y\ge x} \EE_y(e^{\lambda' \T_x})<\infty.
$$
In particular, $\EE_\infty(e^{\lambda' \T_x})=\sup_{y\ge x} \EE_y(e^{\lambda' \T_x})<\infty$, which  gives the needed uniform integrability.
In what follows we consider $\lambda_{1}(z) <\lambda'<\lambda_{1}(x)$.
\smallskip

For a large $M> x$, the It\^o's formula shows that 
$(e^{\lambda_{1}(z) (t\wedge \T_z\wedge \T_M)} \psi_{z,1}(X_{t\wedge \T_z\wedge \T_M}):\, t\ge 0)$ is a martingale
and therefore for every $x<w<M$ we get (by Doob's sampling Theorem)
$$
\begin{array}{l}
\psi_{z,1}(w)=
\EE_w\left(e^{\lambda_{1}(z) (\T_x\wedge \T_M)} \psi_{z,1}(X_{\T_x\wedge \T_M})\right)\\
=\psi_{z,1}(M) \EE_w\left(e^{\lambda_{1}(z) \T_M}, \T_M<\T_x\right)+
\psi_{z,1} \EE_w\left(e^{\lambda_{1}(z) \T_x}, \T_x<\T_M\right).
\end{array}
$$
The first term tends to zero, as $M\to \infty$,  because: $\psi_{z,1}$ is a bounded function;
$$
\EE_w\left(e^{\lambda_{1}(z) \T_M}, \T_M<\T_x\right)\le \EE_w\left(e^{\lambda' \T_x}, \T_M<\T_x\right)
$$
and $\PP_w\left(\T_M<\T_x\right)\to 0$ (notice that $e^{\lambda' \T_x}$ is integrable).
The second term converges to $\psi_{z,1}(x) \EE_w\left(e^{\lambda_{1}(z) \T_x}\right)$ by the Monotone
Convergence Theorem and the result is shown
for finite $w$. Using again that $\psi_{z,1}(\bullet)$ is increasing and bounded, we can pass to the limit $w\to \infty$ to include
this case as well
$$
\psi_{z,1}(\infty)=
\psi_{z,1}(x) \EE_\infty\left(e^{\lambda_{1}(z) \T_x}\right).
$$

\end{proof}

\begin{proof}[Proof of Theorem \ref{the:3}] 
We shall give a proof of $(i)$ for the extreme measure
$\eta=\delta_\infty$.  
For the sake of simplicity we take $z=0$ and we denote by 
$\underline\lambda=\lambda_{1}(0), \underline\psi=\psi_{0,1}$.
For any $x>0$ we have that $\lambda_{1}(x)>\underline \lambda$. As we have seen, $\HH_1$ and $ \HH_2$
imply that for any $\lambda'<\lambda_{1}(x)$
$$
C(\lambda')=\EE_\infty(e^{\lambda' \T_x})=\sup_{w\ge x} \EE_w(e^{\lambda' \T_x})<\infty,
$$
and we conclude that 
$$
\limsup\limits_{t\to \infty} e^{\lambda' t}\, \PP_\infty(\T_x>t)\le C(\lambda').
$$
Now, we fix $\lambda'$ such that $\underline\lambda<\lambda'<\lambda_{1}(x)$ and we use the strong Markov property to get
$$
\PP_\infty(\T_0>t)=\PP_\infty(\T_x>t)+\int_0^t \PP_x(\T_0>u) \, \PP_\infty( \T_x \in d(t-u)).
$$
On the other hand, for every $\epsilon>0$, we have the bound 
$$
\int_0^{\epsilon t} \PP_x(\T_0>u) \, \PP_\infty( \T_x \in d(t-u))\le \PP_\infty( \T_x >(1-\epsilon)t)=\mathcal{O} (e^{-\lambda' (1-\epsilon) t}).
$$
So, if we take $\epsilon=\epsilon(x)>0$ small enough such that $\underline\lambda<\lambda' (1-\epsilon)$, we get 
$$
\int_0^{\epsilon t} \PP_x(\T_0>u) \, \PP_\infty( \T_x \in d(t-u))=o(e^{-\underline\lambda t}).
$$
We also have that $\PP_\infty(\T_x>t)=o(e^{-\underline\lambda t})$.

Now, we use Theorem 5.2  in  \cite{Cattiaux} that gives
$$
\lim\limits_{s\to \infty} e^{\underline\lambda s}\,  \PP_x(\T_0>s)=
\underline\psi(x)\int_0^\infty \underline\psi(y)\, 2e^{-\gamma(y)}\,dy.
$$
Thus, for every $\delta>0$ there exists $t_0=t_0(x)$, such that for all $s>t_0$
$$
\left|\frac{ \PP_x(\T_0>s)}{e^{-\underline\lambda s} \underline\psi(x)\int_0^\infty \underline\psi(y)\, 2e^{-\gamma(y)}\,dy}-1\right|\le \delta.
$$
Hence, if $\,\epsilon t>t_0$ we have
$$
\begin{array}{l}
\int_{\epsilon t}^t \PP_x(\T_0>u) \, \PP_\infty( \T_x \in d(t-u))\\ 
\\
\le(1+\delta)\left(\underline\psi(x)\int_0^\infty \underline\psi(y)\, 2e^{-\gamma(y)}\,dy\right)
\int_{\epsilon t}^t  e^{-\underline\lambda u} \, \PP_\infty( \T_x \in d(t-u))\\
\\
=2e^{-\underline\lambda t} (1+\delta)\int_0^\infty \underline\psi(y) e^{-\gamma(y)}\,dy\, 
\left(\underline\psi(x) \int_0^{(1-\epsilon) t}  e^{\underline\lambda u} \, \PP_\infty( \T_x \in du)\right).
\end{array}
$$
The conclusion is 
$$
\limsup\limits_{t\to \infty} e^{\underline\lambda\, t}\,  \PP_\infty(\T_0>t)\le 
2(1+\delta)\int_0^\infty \underline\psi(y) e^{-\gamma(y)}\,dy \,
\left( \underline\psi(x)\EE_\infty(e^{\underline\lambda \T_x}) \right).
$$
We can take $\delta\downarrow 0$ and use 
$\underline\psi(x)\EE_\infty(e^{\underline\lambda \T_x})=\underline\psi(\infty)$
(see Lemma \ref{lem:0}) to get
$$
\begin{array}{lll}
\limsup\limits_{t\to \infty} e^{\underline\lambda\, t}\,  \PP_\infty(\T_0>t)
&\hspace{-0.2cm}\le&\hspace{-0.2cm} 2\underline\psi(\infty) \int_0^\infty \underline\psi(y) e^{-\gamma(y)}\,dy\\ 
\\
&\hspace{-0.2cm}=&\hspace{-0.2cm}\int_0^\infty \underline\psi(x) \, \delta_\infty(dx) \int_0^\infty \underline\psi(y) e^{-\gamma(y)}\,dy. 
\end{array}
$$
The same lower bound is obtained for the $\liminf$ and $(i)$ is shown.

\medskip

Part $(ii)$ of Theorem \ref{the:3} is shown in the same way, using
again Theorem 5.2 in \cite{Cattiaux}. 
\end{proof}

\smallskip

\begin{remark} 
Obviously we have $\PP_{\rho} (\T_z>t)\le \PP_\infty(\T_z>t)$ and we
can show that both tail distributions are equivalent, namely 
$$
\PP_{\rho} (\T_z>t) \le \PP_\infty(\T_z>t)\le \m(z) \PP_{\rho} (\T_z>t-1).
$$
For this purpose  we use formula (\ref{eq:general}), with the function
$f(T)=(T-t)_+$, we obtain 
$$
\EE_\infty((\T_z-t)_+)=2 \int_z^\infty e^{\gamma(y)} \int_y^\infty
\EE_\xi(\T_z>t) \, e^{-\gamma(\xi)} \, d\xi dy. 
$$
This relation can be written as $\EE_\infty((\T_z-t)_+)= \m(z)
\PP_{\rho} (\T_z>t)$,  
where $\rho$ is the probability measure whose density is
$\frac{1}{\m(z)} \int_z^x e^{\gamma(y)-\gamma(z)} dy \, \ind_{x\ge z}$, with respect to the speed 
measure $2 e^{-(\gamma(x)-\gamma(z))} \,dx$. 
In this way we get 
$$
\PP_\infty(\T_z>t+1)\le \EE_\infty((\T_z-t)_+)= \m(z) \PP_{\rho} (\T_z>t).
$$
\end{remark}

\smallskip

\bigskip

\section{Proofs of Theorems \ref{densite} and \ref{ratio}}

Recall that (see \cite{Cattiaux}) for finite $z>0$, we have  
\be \label{densite}\PP_z(X_t\in dx, \T_0>t)=\Gamma(t,z,x)dx=r(t,z,x)\, 2e^{-\gamma(x)} dx\ee with
$$
r(t,z,x)=\sum\limits_{k\ge 1} e^{-\lambda_k t} \psi_k(z) \psi_k(x),
$$
where $\lambda_k=\lambda_{0,k}$ and $\psi_k=\psi_{0,k}$.
This series is convergent in $L^2(d\mu)$ for all $t>0$ and
$$
\|r(t,z,\bullet)\|^2_{L^2}=\int_0^\infty r(t,z,x)^2 \, 2e^{-\gamma(x)}\, dx=\sum\limits_{k\ge 1} e^{-2\lambda_k t} \psi_k(z)^2<\infty.
$$

So, a natural candidate for a density of $\PP_\infty$ is the series
$$
r(t,\infty,x)=\sum\limits_{k\ge 1} e^{-\lambda_k t} \psi_k(\infty) \psi_k(x).
$$
The quantities $\psi_{k}(\infty)$ are well defined, see Proposition
\ref{prop:b2}(ii).   The first thing to show is that this series converges in $L^2$, that is, 
$$\sum\limits_{k\ge 1} e^{-2\lambda_k t} \psi_k(\infty)^2<\infty.$$
Then, we shall show that $\,r(t,\infty,\bullet)$ is in fact the desired density.

It is clear that we require a control on the growth of $(\psi_k(\infty))_k$ and more generally we need a control
on $(\|\psi_k\|_\infty)_k$. Theorem  \ref{densite} provides this
control
 using the extra hypothesis $\HH_{3}$ on $q$.

\noindent
\begin{proof}[ Proof of Theorem \ref{densite}]  
Part  $(i)$ is an immediate consequence of the bounds provided by 
Proposition \ref{pro:bound4r} and Corollary \ref{cor:2}.
The fact, that $r(t,z,\bullet)$ and $r(t,\infty,\bullet)$ are bounded continuous functions
follows from the uniform convergence (in the $x$ variable) of  the partial sums 
$$
\sum_{k=1}^n e^{-\lambda_k t} \psi_k(z) \psi_k(x) \hbox{ and } 
\sum_{k=1}^n e^{-\lambda_k t} \psi_k(\infty) \psi_k(x).
$$
On the other hand
$$
\sup_{y\ge z; x\ge 0} |r(t,\infty,x)-r(t,y,x)|\le 
\sum_k e^{-\lambda_k t} \|\psi_k\|_\infty \sup_{y\ge z}|\psi_k(\infty)-\psi_k(y)| \to 0
$$
as $z\to \infty$. Hence $r(t,\infty,\bullet)$ is a bounded continuous function and the convergence
of $r(t,z,\bullet)$ to $r(t,\infty,\bullet)$ holds in all $L^p$.

\medskip

\noindent$(ii)$ Recalling  the monotone convergence of $X^z$ to $X^{\infty}$, we have
\ben
&&\mathbb E_\infty(f(X_t), \T_0>t)=\lim\limits_{z\to \infty} \mathbb E_z(f(X_t), \T_0>t)\\
&& \qquad =
\lim\limits_{z\to \infty} \int_0^\infty f(x) r(t,z,x) \, 2e^{-\gamma(x)} dx
=\int_0^\infty f(x) r(t,\infty,x) \, 2e^{-\gamma(x)} dx
\een
for any non-decreasing function and then for any bounded and continuous function $f$. 
This shows that $ r(t,\infty,x) \, 2e^{-\gamma(x)} dx$ is a density for the measure $\PP_\infty(X_t\in \bullet \, ; \, \T_0>t)$.
The rest follows directly.
\end{proof}

\begin{remark} Notice that $\|r(t,\infty,\bullet)\|^2_{L^2}=\sum_k e^{-2\lambda_k t} \psi^2_k(\infty)$.
\end{remark}

As consequence of Corollary \ref{cor:2}, the function $r(t,\bullet,\bullet)$ is uniformly continuous in $\overline\RR_+^2$.
Indeed, we have
$$
\begin{array}{l}
|r(t,x,y)-r(t,x',y')|\le \sum_k e^{-\lambda_k t} \|\psi_k\|_\infty (|\psi_k(x)-\psi_k(x')| + |\psi_k(y)-\psi_k(y')|)\\
\le A \sum_k e^{-\lambda_k t} \|\psi_k\|^2_\infty \lambda_k\,\, (|\int_{x}^{x'} \frac{1}{q(u)+B} du| + 
|\int_{y}^{y'} \frac{1}{q(u)+B} du|).
\end{array}
$$
An important conclusion is that $$\mathscr{H}(z,y):=\inf\limits_{[z,\infty]\times[y,\infty]} r(t,\bullet,\bullet)>0,$$ for $z>0, y>0$. Indeed, by continuity
we have $\mathscr{H}(z,y)=r(t,x^*,u^*)$ for some $(x^*,u^*)\in [z,\infty]\times[y,\infty]$. If $x^*$ and $u^*$ are finite, the
positivity of $r(t,x^*,u^*)$ is a consequence of Maximum Principle. Indeed  $\Gamma(\bullet,\bullet,x)$ defined in \eqref{densite} is a non-negative solution of the heat equation $\LL v=\frac{\partial}{\partial t} v$.
If $x^*=\infty$  but
$u^*$ finite, the function $r(\bullet, \infty,\bullet)$ also satisfies a heat equation and we conclude similarly.  The symmetry of $r$ yields then the result 
 in the case
$x^*$ finite and $u^*=\infty$. Finally, the last case  $x^*=u^*=\infty$  is obtained from
$r(t,\infty,\infty)=\sum_k e^{-\lambda_k t} \psi_k^2(\infty)>0$. 

\medskip

Now, we control the behavior of $r(t,z,x)$ for $x$ near $0$. Of course we have $r(t,z,0)=0$.
Let us prove that $\frac{\partial}{\partial x} r(t,z,0)>0$ for all $t>0,z>0$. Indeed,
\ben
\frac{\partial}{\partial x} r(t,z,0)&=\sum_k e^{-\lambda_k t} \psi_k(z)\psi_k'(0)=
2 \int_0^\infty \sum_k \lambda_k e^{-\lambda_k t} \psi_k(z)\psi_k(\xi) e^{-\gamma(\xi)} \, d\xi\\
&=-\frac{\partial}{\partial t} \int_0^\infty \sum_k e^{-\lambda_k t} \psi_k(z)\psi_k(\xi) 2e^{-\gamma(\xi)} \, d\xi
=-\frac{\partial}{\partial t} \PP_z(\T_0>t)\ge 0.
\een
The last inequality holds, because $\PP_z(\T_0>t)$ is decreasing. Actually, we have proved that
$\frac{\partial}{\partial x} r(t,z,0)$ is the density of $\T_0$ starting from $z$ (a relation that holds under quite general conditions).
On the other hand, $v(t,z)=\frac{\partial}{\partial x} r(t,z,0)\ge 0$ satisfies the heat equation $\LL v=\frac{\partial}{\partial t} v$,
with boundary condition $v(t,0)=0$, for all $t>0$. Therefore, Maximum Principle proves
that $\frac{\partial}{\partial x} r(t,z,0)>0$ on $(0,\infty)^2$.

\medskip

Similarly, $\frac{\partial}{\partial x} r(t,\infty,0)$ is the density of $\T_0$ under $\PP_\infty$. So, we conclude it is not
negative and we prove now it is strictly positive. For that purpose consider $0<s<t$ and recall \eqref{eq:13}, so
$$
\PP_\infty(\T_0>t)=\EE_\infty(\mathbb P_{X_{s}} (\T_0>t-s);  \T_0>s )
=\int_0^\infty r(s,\infty,z) 2e^{-\gamma(z)} \PP_z(\T_0>t-s) \, dz.
$$
The conclusion is $\frac{\partial}{\partial x} r(t,\infty,0)=\int_0^\infty r(s,\infty,z) 2e^{-\gamma(z)} \frac{\partial}{\partial x} r(t-s,z,0) \, dz>0$.

\medskip

\begin{proof}[Proof of Theorem \ref{ratio}]

 We notice that for $z\ge z_0$, $y_0>0$ and some finite constants $A,B$ (see Corollary \ref{cor:2})
\ben
\sup\limits_{x\ge y_0; w\ge z} \left|\frac{r(t,\infty,x)}{r(t,w,x)}-1\right| &\le & \frac{1}{\mathscr{H}(z_0,y_0)}  
\sup\limits_{x\ge y_0, w\ge z} |r(t,\infty,x)-r(t,w,x)|
\\
&\le& \frac{A}{\mathscr{H}(z_0,y_0)} \sum_k e^{-\lambda_k t} \|\psi_k\|^2_\infty \lambda_k\, \int_{z}^\infty \frac{1}{q(u)+B} du,
\een
which converges to $0$ as $z$ converges to $\infty$, recalling \eqref{convint}.

To finish the proof we need to bound, for some $y_0>0$ and large $z$, the ratio
$$
\sup\limits_{x\le y_0; w\ge z} \left|\frac{r(t,\infty,x)}{r(t,w,x)}-1\right|.
$$
Let us estimate, for $w\wedge w'\ge z, 0\le x\wedge x'\le x\vee x' \le y_0$,
$$
\begin{array}{l}
|\sum_k e^{-\lambda_k t} \psi_k(w) \psi_k'(x)-\sum_k e^{-\lambda_k t} \psi_k(w') \psi_k'(x')|\le\\
\\
\sum_k e^{-\lambda_k t} |\psi_k(w) -\psi_k(w')| \max\limits_{0\le y\le y_0}| \psi_k'(y)|+
\sum_k e^{-\lambda_k t} |\psi_k|_\infty |\psi_k'(x)-\psi_k'(x')|.
\end{array}
$$
Letting $0\le x\le x'\le y_0$ and using the expression of $\psi_k$ given in (C.2),
$$
\begin{array}{l}
|\psi_k'(x)-\psi_k'(x')|=2\lambda_k e^{\gamma(x)} \left|\int_x^\infty e^{-\gamma(\xi)}
\psi_k(\xi)\, d\xi-e^{\gamma(x')-\gamma(x)} \int_{x'}^\infty e^{-\gamma(\xi)}
\psi_k(\xi)\, d\xi \right|\\
\\
\le \lambda_k e^{\gamma(x)} |e^{\gamma(x')-\gamma(x)}-1| \int_0^\infty 2e^{-\gamma(\xi)} |\psi_k(\xi)|\, d\xi
+\lambda_k e^{\gamma(x)}  \int_x^{x'} 2e^{-\gamma(\xi)} |\psi_k(\xi)|\, d\xi\\
\\
\le C \lambda_k |x-x'|,
\end{array}
$$
where  $C=C(y_0)$.
Similarly, $|\psi_k'(y)|\le D \lambda_k$, for $0\le y\le y_0$ and some finite constant $D=D(y_0)$.
We use Proposition \ref{pro:bound4r} and Corollary \ref{cor:2} with $\epsilon=t/4$ to get, for $w\wedge w'\ge z, 0\le x\wedge x'\le x\vee x'\le y_0$
$$
\left|\frac{\partial}{\partial x} r(t,w,x)-\frac{\partial}{\partial x} r(t,w',x')\right|\le F \left(\int_w^\infty \frac{1}{q(u)+B} du + |x-x'|\right),
$$
for some finite constant $F=F(t,y_0)$ (notice that $F$ gets larger as $t$ gets smaller). We use this bound for $w'=\infty, x'=0, y_0\le 1$ to get
$$
\left|\frac{\partial}{\partial x} r(t,w,x)-\frac{\partial}{\partial x} r(t,\infty,0)\right|\le \overline F \left(\int_w^\infty \frac{1}{q(u)+B} du +  y_0\right)
$$
and $\overline F=F(t,1)$. Hence, if $z_0$ is large enough and  $y_0$ is small enough (depending on $t$), we get
$\overline F \left(\int_{z_0}^\infty \frac{1}{q(u)+B} du + y_0 \right)\le \frac{1}{2} \frac{\partial}{\partial x} r(t,\infty,0)$ and then
for all $w\ge z\ge z_0, 0\le \xi \le y_0$
$$
\frac{\partial}{\partial x} r(t,w,\xi)\ge \frac{1}{2} \frac{\partial}{\partial x} r(t,\infty,0)>0.
$$
This bound gives the inequality
\ben
\sup\limits_{w\ge z; 0\le x\le y_0} \left|\frac{r(t,\infty,x)}{r(t,w,x)}-1\right| 
\le \sup\limits_{w\ge z; 0\le \xi\le y_0} \left|\frac{\frac{\partial}{\partial x} r(t,\infty,\xi)}{\frac{\partial}{\partial x} r(t,w,\xi)}-1\right| \qquad \qquad \qquad \qquad &&\\
 \quad \le  \frac{\sup\limits_{w\ge z; 0\le \xi \le y_0} |\frac{\partial}{\partial x} r(t,\infty,\xi)-\frac{\partial}{\partial x} r(t,w,\xi)|}
{\inf\limits_{w\ge z_0; 0\le \xi \le y_0} \frac{\partial}{\partial x} r(t,w,\xi)}
\le \frac{ 2 \overline F}{ \frac{\partial}{\partial x} r(t,\infty,0)} \int_{z}^\infty \frac{1}{q(u)+B} du .&&
\een
The result is shown.
\end{proof}
\begin{remark} We have proved: For all $t>0$ there exist $z_0=z_0(t)>0, G=G(t)<\infty$ such that for all $z\ge z_0$,
$$
\sup\limits_{x\ge 0} \left|\frac{r(t,\infty,x)}{r(t,z,x)}-1\right|\le G \int_{z}^\infty \frac{1}{q(u)} du.
$$
\end{remark}

\appendix
\section{First basic estimations}
\label{App:A}

In this section we shall get some basic asymptotic relations among different quantities that we use in the rest of the article.
%We assume through this section that $\HH_1,\HH_2$ holds. 
The first tool is a consequence of L'H\^opital's rule.
\smallskip

\begin{lemma} 
\label{lem:1}
Assume that $H$ is an integrable (eventually) strictly decreasing positive function that
satisfies $$\lim\limits_{z\to \infty} \frac{H^2(z)}{H'(z)}=0 \quad  \text{and} \quad \lim\limits_{z\to \infty} \frac{H(z)H''(z)}{(H'(z))^2}=a\in \overline\RR\setminus\{2\},$$
 then
$$
\lim\limits_{z\to \infty} \frac{\int_z^\infty H(y) \, dy}{\frac{H^2(z)}{-H'(z)}}=\frac{1}{2-a}.
$$
For $a=2$, the result holds  assuming also
$\frac{H(z)H''(z)}{(H'(z))^2}\neq 2$ for  $z$ large enough.
\end{lemma}
\begin{proof} Using L'H\^opital's rule, we get
$$
\frac{\int_z^\infty H(y) \, dy}{\frac{-H^2(z)}{H'(z)}} \rightsquigarrow \frac{-H}{\frac{-2H(H')^2+H^2H''}{(H')^2}}=
 \frac{1}{2-\frac{HH''}{(H')^2}}
$$
and the lemma is shown.
\end{proof}

\medskip
The following result is an application of this Lemma
\smallskip

\begin{corollary} 
\label{cor:1}
Assume that $\q$ satisfies $\HH_1$ and $\HH_2$. Then, the following limits hold
\begin{align}
(i) \qquad &\lim\limits_{z\to \infty} q(z)\m'(z)= \lim\limits_{z\to \infty} 2\q(z) e^{\gamma(z)} \int_z^\infty e^{-\gamma(\xi)} \,d\xi =1,  \\
(ii) \qquad &\lim\limits_{z\to \infty} \frac{\m(z)}{M(z)}= \lim\limits_{z\to \infty} \frac{\int_z^\infty 2e^{\gamma(y)} \int_y^\infty e^{-\gamma(\xi)} \,d\xi \,dy}
{\int_z^\infty\frac{1}{\q(y)}\, dy}=1,  \\
(iii) \qquad &\lim\limits_{z\to \infty} \frac{\int_z^\infty e^{\gamma(y)} \left(\int_y^\infty e^{-\gamma(\xi)} \,d\xi \right)^2\,dy} 
{e^{2\gamma(z)} \left(\int_z^\infty e^{-\gamma(\xi)} \,d\xi \right)^3}=1, \\
(iv) \qquad &\lim\limits_{z\to \infty} \frac{8\int_z^\infty e^{\gamma(y)} \int_y^\infty e^{\gamma(\eta)} 
\left(\int_{\eta}^\infty e^{-\gamma(\xi)} \,d\xi\right)^2\, d\eta\, dy}
{\int_z^\infty \frac{1}{\q^3(y)} \, dy}=1. 
\end{align}
\end{corollary}
\begin{proof} ($i$) Notice that $H(z)=e^{-\gamma(z)}$ is positive and strictly decreasing. Also, it satisfies 
$\lim\limits_{z\to \infty} \frac{H^2(z)}{H'(z)}=0$. 
On the other hand
$$
\frac{H''(z)H(z)}{(H'(z))^2}=\frac{(4\q^2(z)-2\q'(z))e^{-2\gamma(z)}}{4\q^2(z)e^{-2\gamma(z)}}\underset{z \to \infty}{\longrightarrow}1,
$$
because $\HH_2$. Hence, Lemma \ref{lem:1} implies that
\begin{equation}
\label{eq:3.0}
\lim\limits_{z\to \infty} \frac{\int_z^\infty e^{-\gamma(\xi)} \,d\xi}{\frac{e^{-2\gamma(z)}}{2\q(z)e^{-\gamma(z)}}}=1.
\end{equation}
($ii$) is a consequence of ($i$) by L'H\^opital's rule.

\medskip

\noindent($iii$) Consider $H(z)=e^{\gamma(z)} \left( \int_{z}^\infty e^{-\gamma(\xi)}\, d\xi\right)^2$, 
whose derivative is
$$
\begin{array}{ll}
H'(z)&=2\q(z) e^{\gamma(z)} \left( \int_{z}^\infty e^{-\gamma(\xi)}\, d\xi\right)^2-2 \int_{z}^\infty e^{-\gamma(\xi)}\, d\xi\\
\\
&=\left(2\q(z) e^{\gamma(z)} \int_{z}^\infty e^{-\gamma(\xi)}\, d\xi-2\right) \int_{z}^\infty e^{-\gamma(\xi)}.
\end{array}
$$
Therefore, $H$ is eventually strictly decreasing. On the other hand 
{\small
$$
\frac{H(z)H''(z)}{(H'(z))^2}=\frac{\left[(4\q^2(z)+2\q'(z)) e^{2\gamma(z)} \left(\int_{z}^\infty 
e^{-\gamma(\xi)}\right)^2-4\q(z) e^{\gamma(z)} \int_{z}^\infty e^{-\gamma(\xi)}\, d\xi+2\right] }
{\left(2\q(z) e^{\gamma(z)} \int_{z}^\infty e^{-\gamma(\xi)}\, d\xi-2\right)^2}
\to 1.
$$
}
Hence, we need to compute 
$$
\frac{H^2(z)}{H'(z)}=\frac{e^{2\gamma(z)}\left( \int_{z}^\infty e^{-\gamma(\xi)}\, d\xi\right)^4}
{\left(2\q(z) e^{\gamma(z)} \int_{z}^\infty e^{-\gamma(\xi)}\, d\xi-2\right)\int_{z}^\infty e^{-\gamma(\xi)}\, d\xi}
\rightsquigarrow -e^{2\gamma(z)}\left( \int_{z}^\infty e^{-\gamma(\xi)}\, d\xi\right)^3,
$$
which by $(i)$ converges to $0$ and ($iii$)  is shown.
\medskip

\noindent ($iv$) is a consequence of ($i$) and ($iii$).
\end{proof}

\begin{lemma}
\label{nouveau}
Assume $\HH_{1}$, $\lim\limits_{x\to\infty}q(x)=\infty$   and \eqref{eq:H2.5.2}. Then
\begin{enumerate}[(i)]
\item For any $y\in\RR_+$,
$$
\lim\limits_{z\to\infty}\sup_{|x-z|\le y\sqrt{\m(z)}}\big|\log\big(q(z)/q(x)\big)\big|=0.
$$
\item For any $y\in\RR$,
$$
\lim\limits_{z\to\infty}q(z)\;
\frac{\int_{z}^{z+y\,\sqrt{\m(z)}}dx/q(x)}{\sqrt{\m(z)}}=y. $$
\end{enumerate}
\end{lemma} 
\begin{proof}We first give an upper and a lower
bound for $q(z)/q(x)$ for large $z$ and $|x-z|<y\m(z)$. Note
that if $y$ is fixed and $z$ tends to infinity, any point
$\xi\in[z-y\m(z),z+y\m(z)]$ also tends to infinity.

We have
$$
\log\big(q(z)/q(x)\big)=\int_{x}^{z}\frac{q'(u)}{q(u)}\;du\;.
$$
Therefore by the mean value Theorem there exists $\xi\in[x,z]$ (if
$x<z$, $\xi\in[z,x]$ if $x>z$) such that
$$
\log\big( q(z)/q(x)\big)=(z-x)\;\frac{q'(\xi)}{q(\xi)}
=\left(\frac{z-x}{\sqrt{M(\xi)}}\right)\;
\left(\frac{q'(\xi)}{q(\xi)}\,\sqrt{M(\xi)}\right)\;.
$$ 
By condition  \eqref{eq:H2.5.2}  the second term tends to zero when $z$ (and
hence $x$ and $\xi$) tends to infinity and we now prove that the first factor
is of bounded modulus. We have for any   $x \in [z- y\sqrt{\m(z)}, z+y\sqrt{\m(z)}]$,
$$
\left|\frac{z-x}{\sqrt{M(\xi)}}\right|\le y \sqrt{\frac{\m(z)}{M(\xi)}}\leq  
y \sqrt{\frac{\m(z)}{M(z-y\sqrt{\m(z)})}}
$$ 
since $M$ is decreasing.  Recalling that $\HH_1$ and $\HH_{2}$ hold, by Corollary \ref{cor:1} $(ii)$
it is enough to prove that
$$
\lim\limits_{z\to\infty}\frac{M(z)}{M(z-y\sqrt{\m(z)})}=1.
$$
For that purpose, we  observe that
$$\bigg\vert \frac{M(z-y\sqrt{\m(z)})}{M(z)}-1\bigg\vert=\frac{\int_{z-y\sqrt{\m(z)}}^z du/q(u)}{M(z)}\leq y\frac{\sqrt{\m(z)}}{M(z)}\sup_{ [z-y\sqrt{\m(z)},z]} \frac{1}{q}.$$
Using  again Corollary \ref{cor:1} $(ii)$ and the monotonicity of $M$, there exists $C>0$ such that
\ben
\bigg\vert \frac{M(z-y\sqrt{\m(z)})}{M(z)}-1\bigg\vert 
&\leq &  \frac{C}{\sqrt{M(z)}}\sup_{ [z-y\sqrt{\m(z)},z]} \frac{1}{q}
\leq C\sup_{ [z-y\sqrt{\m(z)},z]}  \frac{1}{\sqrt{M}q}
\een
and the last term goes to $0$ as $z\rightarrow \infty$ by Lemma \ref{conhypop}$(i)$.
We obtain $(i)$, which  immediately implies (\textit{ii}).
\end{proof}

\bi Recall notation $\sigma(z)=\sqrt{\hbox{Var}(\T_z^{(\infty)})}$,  $r_y(t)=\m^{-1}(t)+y\sqrt t$ and $r_{y,\epsilon}(t)=r_y(t)+\epsilon\sqrt t$
for  $t\in\RR_{+}$. 
\begin{lemma} \label{tec} Assume $\HH_{1}$, $\lim\limits_{x\to\infty}q(x)=\infty$,  \eqref{eq:H2.5.1}   and \eqref{eq:H2.5.2}. \\
\ben
(i) &\qquad
\lim\limits_{t\searrow 0}\, \frac{t- \m(r_{y}(t))}{\sigma(r_{y}(t))}=\sqrt{\Sigma}\, y.\\
(ii) & \qquad \frac{1}{4}\le \liminf\limits_{x\to\infty}
\frac{\Lambda(x)}{e^{\gamma(x)}/q(x)}
\le \limsup\limits_{x\to\infty}
\frac{\Lambda(x)}{e^{\gamma(x)}/q(x)}\le 1.\\
(iii) &  \qquad
\lim\limits_{t\to\infty}\frac{q\big(r_{y,\epsilon}(t)\big)}
{q\big(r_{y}(t)\big)}=1.\\
(iv) & \qquad 
\lim\limits_{t\searrow 0}\, \frac{e^{\gamma(r_{y}(t))}}{e^{\gamma(r_{y,\epsilon}(t))}}=0.
\een
\end{lemma}
\begin{proof}
From \eqref{lyapounov}  and Theorem \ref{mom} (\textit{i}), we have
$$
\m(r_{y}(t))=\EE\big(\T_{r_{y}(t)}^{(\infty)}\big)
=t-2\int_{\m^{-1}(t)}^{\m^{-1}(t)+y\,\sqrt t} e^{\gamma(u)}du\int_{u}^{\infty}e^{-\gamma(\xi)} d\xi
$$
and we have to compute
$$
\lim\limits_{t\searrow 0}\frac{\int_{\m^{-1}(t)}^{\m^{-1}(t)+y\,\sqrt t}
e^{\gamma(u)}du\int_{u}^{\infty}e^{-\gamma(\xi)} d\xi}{
\sigma(r_{y}(t))}
=\lim\limits_{z\to\infty}
\frac{\int_{z}^{z+y\,\sqrt{\m(z)}}
e^{\gamma(u)}du\int_{u}^{\infty}e^{-\gamma(\xi)} d\xi}{\sigma(z+y\,\sqrt{\m(z)}))}\;.
$$
Let us prove now that 
\be
\label{eq:rapvar}
\lim\limits_{z\to \infty} \frac{\sigma(z+y\,{\sqrt{\m(z)}})}{\sigma(z)}=1.
\ee
Using \eqref{varex} it is equivalent to prove
$$
\lim\limits_{z\to \infty} \frac{\int_{z+y\,{\tiny\sqrt{\m(z)}}}^\infty q^{-3}(x)\, dx}{\int_z^\infty q^{-3}(x)\, dx}=1.
$$
Using l'H\^opital's rule we need to study the ratio
$$
\left(1+y\frac{\m'(z)}{2\sqrt{\m(z)}}\right)\left(\frac{q(z)}{q(z+y\,{\tiny\sqrt{\m(z)}})}\right)^3.
$$
Using respectively Lemma \ref{conhypop} $(iii)$ and  Lemma \ref{nouveau} $(i)$, both factors converge to 1, when $z\to \infty$. This shows
\eqref{eq:rapvar} and therefore
$$
\lim\limits_{t\searrow 0}\frac{\int_{\m^{-1}(t)}^{\m^{-1}(t)+y\,\sqrt t}
e^{\gamma(u)}du\int_{u}^{\infty}e^{-\gamma(\xi)} d\xi}{
\sigma(r_{y}(t))}
=\lim\limits_{z\to\infty}
\frac{\int_{z}^{z+y\,\sqrt{\m(z)}}
e^{\gamma(u)}du\int_{u}^{\infty}e^{-\gamma(\xi)} d\xi}{
\sigma(z)}\;.
$$
Using again Corollary \ref{cor:1} (\textit{i}) and relation 1.5
this last limit is equal to
$$
\begin{array}{l}
\lim\limits_{z\to\infty}\frac{\int_{z}^{z+y\,\sqrt{\m(z)}}q(x)^{-1}\;dx}
{2\,\sqrt{\int_{z}^{\infty}q(x)^{-3}\:dx}}
=\lim\limits_{z\to\infty}\frac{q(z)\;
\int_{z}^{z+y\,\sqrt{\m(z)}}q(x)^{-1}\;dx}{2\,\sqrt{\m(z)}}\;
\frac{\sqrt{\m(z)}}{q(z)\;\sqrt{\int_{z}^{\infty}q(x)^{-3}\:dx}}
=\frac{\sqrt{\Sigma} \, y}{2}\;.
\end{array}
$$
The first factor in the above expression, converges to $y$ by  Lemma \ref{nouveau} (\textit{ii}). On the other
hand, the second factor converges to  $\sqrt{\Sigma}$  by hypothesis \eqref{eq:H2.5.1} and the fact that 
$\m(z)/M(z) \to 1$ (Corollary \ref{cor:1} (\textit{ii})).  \\

Let us now prove $(ii)$.
From $\HH_{1}, \HH_{2}$, we can choose $A\in (0,\infty)$, such that
$$
\inf_{x\ge A}q(x)>1\;, \qquad \sup_{x\ge A}\left|\frac{q'(x)}{q^{2}(x)}\right|<\frac{1}{2} \;.
$$
Using integration by parts, we have
$$
\Lambda(x)-\Lambda(A)=\int_{A}^{x}\frac{1}{2\,q(\xi)}\; 2\,q(\xi)\,e^{\gamma(\xi)}d\xi
=\frac{e^{\gamma(x)}}{2\,q(x)}-\frac{e^{\gamma(A)}}{2\,q(A)}
+\int_{A}^{x}\frac{q'(\xi)}{q(\xi)^{2}}\;e^{\gamma(\xi)}d\xi\;.
$$
Therefore from our choice of $A$ we have for $x>A$
$$
\frac{1}{2}\;\left(\frac{e^{\gamma(x)}}{2\,q(x)}-
\frac{e^{\gamma(A)}}{2\,q(A)}\right)
\le \Lambda(x)-\Lambda(A)\le 2\;\left(\frac{e^{\gamma(x)}}{2\,q(x)}-
\frac{e^{\gamma(A)}}{2\,q(A)}\right)\;.
$$
Since $\Lambda(x)$ diverges with $x$ (by $\HH_{1}$), we obtain $(ii)$. \\

Using Lemma \ref{nouveau} (\textit{i}) and replacing $y$ by
$y+\epsilon$ and taking $z=\m^{-1}(t)$, we get $(iii)$.  \\
 
We have from the definition of the function $\gamma$ 
$$
\log\left(\frac{e^{\gamma(r_{y}(t))}}{e^{\gamma(r_{y,\epsilon}(t))}}\right)=
-2\,\int_{r_{y}(t)}^{r_{y,\epsilon}(t)}q(\xi)\;d\xi\;.
$$
Using again Lemma \ref{conhypop} (\textit{iii}), for small $t$, this quantity is
equivalent to
$$
-2\,\epsilon\,\sqrt{t}\;q\big(r_{y}(t)\big)=-2\,\epsilon\,\sqrt{\m(z)}\;q\big(z+y\,\sqrt{\m(z)}\big),
$$ 
by setting $t=\m(z)$. Finally,
$$
\begin{array}{l}
\sqrt{\m(z)}\;q\big(z+y\,\sqrt{\m(z)}\big)
=\frac{\sqrt{\m(z)}}{\sqrt{\int_{z}^{\infty}d\xi/\q(\xi)}}\;
\frac{q\big(z+y\,\sqrt{\m(z)}\big)}{q(z)}\;q(z)\,
\sqrt{\int_{z}^{\infty}d\xi/\q(\xi)}\;,
\end{array}
$$   
 tends to $\infty$, when $z\to \infty$, for any $\epsilon>0$ 
by Corollary  \ref{cor:1} (\textit{ii}), 
Lemma \ref{conhypop} (\textit{iii}) and Lemma \ref{conhypop} (\textit{i}).  It proves $(iv)$.
\end{proof}

\begin{lemma} \label{lip}
Assume $\HH_{1}$, $\lim\limits_{x\to\infty} q(x)=\infty$ and 
$\limsup\limits_{x\to\infty} \frac{|q'(x)|}{q(x)}<\infty.$
Then there exists a constant $\mathfrak{G}>0$ such that
$$
\sup_{x\in[0,\infty)}\frac{\big|\m''(x)\big|}{\m'(x)}\le \mathfrak{G}\;.
$$
\end{lemma}

\begin{proof}
From $\LL\m=1$ we derive
$\frac{1}{2}\,\m''=q\,\m'+1$. 
Moreover 
\ben
\m'(x)&=&-2\,e^{\gamma(x)} \int_x^\infty e^{-\gamma(\xi)} d\xi=
 e^{\gamma(x)}\,\int_x^\infty
 \frac{1}{q(\xi)}\;d\left(e^{-\gamma(\xi)}\right) \\
&=&-\frac{1}{q(x)}+e^{\gamma(x)} \int_x^\infty
\frac{q'(\xi)}{q(\xi)^{2}}\;
e^{-\gamma(\xi)} d\xi\;.
\een
by integration by parts. Therefore
$$
\frac{1}{2}\,\m''(x)= q(x)\;e^{\gamma(x)} \int_x^\infty
\frac{q'(\xi)}{q(\xi)^{2}}e^{-\gamma(\xi)} d\xi$$
and writing $g(x)=e^{\gamma(x)} \int_x^\infty
\frac{e^{-\gamma(\xi)}}{q(\xi)}\;
 d\xi$, there exists 
 $\mathfrak{A}>0$ such that
 $$\big|\m''(x)\big|\leq  \mathfrak{A}\;q(x)\; g(x).$$
 Moreover, by integration by parts,
$$
g(x)=\frac{1}{2\;q^2(x)}-\;e^{\gamma(x)} 
\int_x^\infty\frac{q'(\xi)}{q(\xi)^{3}}\;e^{-\gamma(\xi)} d\xi=
\frac{1}{2\;q^2(x)}-g(x)o(1)$$
as $x\rightarrow \infty$, using that $|q'(x)|/q(x)^2=|q'(x)|/q(x).1/q(x) \rightarrow 0$.
Then $g(x)\sim 1/(2q^2(x))$ and  we get
$$\limsup\limits_{x\to\infty}\frac{|\m''(x)\big|}{\m'(x)}\leq \frac{\mathfrak{A}}{2}\limsup\limits_{x\to\infty} \frac{1}{q(x)\m'(x)}<\infty$$
by Corollary \ref{cor:1}(\textit{i}). The Lemma
follows by observing that in any compact set $|\m''|$ is bounded and
$\m'$ does not vanish.  
\end{proof}

\section{Central moments estimations.}
\label{moments}

In this section we study  the asymptotic behavior for the
moments of the random variable $\T_z$, under 
$\PP_\infty$, assuming that $\HH_1, \HH_2$ holds.
We will derive results stronger that those of Corollary \ref{cor:Pierre}
$(ii)$. 
\smallskip

The following limits exist, for all $n\ge1$ and all $z\ge 0$, 
$$
\EE_\infty(\T_z^n)=\lim\limits_{x\to \infty} \EE_x(\T_z^n)=
2n\int_z^\infty e^{\gamma(y)} \int_y \EE_\xi(\T_z^{n-1}) e^{-\gamma(\xi)} \,d\xi\,dy.
$$
The main difficulty to get the asymptotic behavior, on $z$, for these
moments is that  the dependence
on $z$ in both the limits of integration and the integrand. To
simplify the dependence on the integrand we shall 
prove the following equality.

\begin{proposition}
\label{recurrence}
 For all $n\ge 1$ and all $z<\xi$, we have 
%\begin{mdframed}
\begin{equation}
\label{for:1}
\EE_\xi(\T_z^n)=\sum\limits_{k=0}^n (-1)^k \sum\limits_{\vec\ell \in \Delta_k^n}
\left[\prod\limits_{i=1}^{k} {\ell_{i-1}\choose \ell_i} \EE_{\infty}(\T_\xi^{\ell_{i-1}-\ell_i}) \right]
(\EE_{\infty}(\T_z^{\ell_k})-\EE_{\infty}(\T_\xi^{\ell_k})),
\end{equation}
%\end{mdframed}
where $\Delta_k^n=\{ \vec\ell =(\ell_0,\ell_1,\cdots,\ell_k) \in \NN^{k+1}:\, 0\le \ell_k<\ell_{k-1} < \cdots < \ell_1<n=\ell_0\}$  ( an empty
product is taken to be 1).
\end{proposition}

\begin{proof} Let us consider $z<\xi<x$.
During  the proof we will denote for convenience (when it is necessary)  by $\T_{x\to z}$ the  time for the process to go from $x$ to $z$ for any $x>z$.

The strong Markov property ensures that  $\EE_\xi((\T_z)^n)=\EE_x((\T_z-\T_\xi)^n)$ and therefore
$$
\begin{array}{l}
\EE_\xi(\T_z^n)-\EE_x(\T_z^n)=\sum\limits_{j=0}^{n-1} (-1)^{n-j} {n\choose j} \EE_x(\T_\xi^{n-j}\T_z^j )\\
=\sum\limits_{j=0}^{n-1} (-1)^{n-j} {n\choose j} \EE_x\left(\T_\xi^{n-j}(\T_{x\to \xi}+\T_{\xi\to z})^j\right)\\
=\sum\limits_{j=0}^{n-1} (-1)^{n-j} {n\choose j} \sum\limits_{r=0}^j {j\choose r} \EE_x\left(\T_\xi^{n-(j-r)}\EE_\xi(\T_z^{(j-r)})\right)\\
=\sum\limits_{j=0}^{n-1} (-1)^{n-j} {n\choose j} \sum\limits_{r=0}^j {j\choose r} \EE_x(\T_\xi^{n-r})\EE_\xi(\T_z^r)
=\sum\limits_{r=0}^{n-1} \EE_x(\T_\xi^{n-r})\EE_\xi(\T_z^r) \sum\limits_{j=r}^{n-1} (-1)^{n-j} {n\choose j} {j\choose r}\\ 
=\sum\limits_{r=0}^{n-1} {n\choose r} \EE_x(\T_\xi^{n-r})\EE_\xi(\T_z^r) \sum\limits_{j'=0}^{n-r-1} (-1)^{n-r-j'} {n-r\choose j'} 
=-\sum\limits_{r=0}^{n-1} {n\choose r} \EE_x(\T_\xi^{n-r})\EE_\xi(\T_z^r).
\end{array}
$$
Therefore, we have proven that
$$\EE_\xi(\T_z^n)=\Big(\EE_x(\T_z^n)-\EE_x(\T_\xi^n)\Big)
-\sum\limits_{r=1}^{n-1} {n\choose r} \EE_x(\T_\xi^{n-r})\EE_\xi(\T_z^r).
$$
Taking  the limit as $x\to \infty$ yields a similar recurrence formula with $\EE_\infty$ instead $\EE_x$.
This recurrence relation can be solved by considering the formal power series $\sum_{n\ge 1} {\EE_\xi(\T_z^n)\over n!} u^n$  (in the variable $u$).  
\end{proof}

\bi
We are now interested in the  moments of the normalized ratio under $\PP_\infty$
$$
\frac{\T_z-\EE_\infty(\T_z)}{\sqrt{\hbox{Var}_\infty(\T_z)}},
$$
where $\hbox{Var}_\infty(\T_z)=\EE_\infty(\T_z^2)-(\EE_\infty(\T_z))^2$.
We start by estimating this quantity
$$
\begin{array}{lll}
\hbox{Var}_\infty(\T_z)&\hspace{-0.2cm}=&\hspace{-0.2cm}4\int_z^\infty e^{\gamma(y)} \int_y^\infty \EE_\xi(\T_z) e^{-\gamma(\xi)} \,d\xi\, dy-(\EE_\infty(\T_z))^2\\
&\hspace{-0.2cm}=&\hspace{-0.2cm}(\EE_\infty(\T_z))^2-4\int_z^\infty e^{\gamma(y)} \int_y^\infty \EE_\infty(\T_\xi) e^{-\gamma(\xi)} \,d\xi\, dy.
\end{array}
$$ 
In particular, we have
$$
(e^{-\gamma(z)} \hbox{Var}_\infty(\T_z)')'=
8 e^{\gamma(z)}\left( \int_z^\infty e^{-\gamma(\xi)} \,d\xi \right)^2.
$$
Integrating this equality yields
$$
 \hbox{Var}_\infty(\T_z)'=8 e^{\gamma(z)} \int_{z}^\infty
 e^{\gamma(\xi)}\left( \int_y^\infty e^{-\gamma(\eta)} \,d\eta
 \right)^2 d\xi,
 $$
and finally
\begin{equation}
\label{eq:2}
\hbox{Var}_\infty(\T_z)=8\,\Psi(z)
\end{equation}
with
$$
\Psi(z)=\int_z^\infty e^{\gamma(y)} \int_y^\infty e^{\gamma(\xi)} 
\left(\int_{\xi}^\infty e^{-\gamma(\eta)} \,d\eta\right)^2\, d\xi\, dy\;.
$$

\me The formula \eqref{varex} follows from \eqref{eq:2} and from Corollary \ref{cor:1} $(iv)$.

\bi

\begin{lemma}[Fourth moment estimation]\label{eq:fourth}
$$
\lim\limits_{z \to \infty}
\frac{\EE_\infty\left([\T_z-\EE_\infty(\T_z)]^4\right)}{(\emph{Var}_\infty(\T_z))^2}=3.  
$$
\end{lemma}
\begin{proof}
We write  $\T=\T_z^{(\infty)}$, $S=\T_\xi^{(\infty)}$, $\m=\EE(\T)$ 
$\hbox{Var}=\hbox{Var}(\T)$.  We also introduce the notation
$$
e^\gamma=e^{\gamma(z)}, \quad \int e^{-\gamma}=\int_z^\infty e^{-\gamma(\xi)}d\xi,  \quad e^\gamma\!\!\int \!\!e^{-\gamma}=e^{\gamma(z)} \int_z^\infty e^{-\gamma(\xi)}d\xi,$$
and
$$
 \iint [\bullet]=\int_z^\infty e^{\gamma(y)} \int_y^\infty \bullet e^{-\gamma(\xi)} \,d\xi\, dy, 
 \quad \int [\bullet]=\int_z^\infty \bullet \, e^{-\gamma(\xi)} \,d\xi. 
 $$
We start by noticing that from (\ref{for:0}) and (\ref{for:1})
$$
\begin{array}{ll}
\EE(\T^4)&=8\iint [\EE_\xi(\T^3)]\\
\\
&=8\iint [\EE(\T^3)-\EE(S^3)-3\EE(S)(\EE(\T^2)-\EE(S^2))-3\EE(S^2)(\EE(\T)-\EE(S))\\
&\hspace{0.3cm}+6\EE(S)^2(\EE(\T)-\EE(S))]\\
\\
&=4\EE(\T^3)\m-24\m \iint [\EE(S^2)-2\EE(S)^2]-24\EE(\T^2)\iint [\EE(S)]+8\iint [-\EE(S^3)\\
&\hspace{0.3cm}+6\EE(S^2)\EE(S)-6\EE(S)^3]\\
\\
&=4\EE(\T^3)\m-24\m \iint [\EE(S^2)-2\EE(S)^2]+6\,\hbox{Var}^2-6\m^4\\
&\hspace{0.3cm}+8\iint [-\EE(S^3)+6\EE(S^2)\EE(S)-6\EE(S)^3].
\end{array}
$$
Here, we have used that $4\iint [\EE(S)]=4\iint [\EE(\T)-\EE_\xi(\T)]=2\m^2-\EE(\T^2)$. With these 
relations we can compute the fourth central moment
$$
\begin{array}{l}
\EE((\T-\m)^4)=\EE(\T^4)-4\EE(\T^3)\m+6\EE(\T^2)\m^2-3\m^4\\
=6\, \hbox{Var}^2-24\m \iint [\EE(S^2)-2\EE(S)^2]+8\iint [-\EE(S^3)+6\EE(S^2)\EE(S)-6\EE(S)^3]-6\m^4\\
\hspace{0.3cm}+6\, (2\m^2-4\iint [\EE(S)])\m^2-3\m^4\\
=6\, \hbox{Var}^2-24\m \iint [\EE(S^2)-2\EE(S)^2]-24\m^2\iint [\EE(S)]
+3\m^4+\\
\;\;8\iint [-\EE(S^3)+6\EE(S^2)\EE(S)-6\EE(S)^3].
\end{array}
$$
We define $Q=Q(z)=\EE_\infty((\T_z-\m)^4)-6\,(\hbox{Var}_\infty(\T_z))^2$. Hence,
$$
\begin{array}{l}
Q'=48 \,e^\gamma\!\!\int \!\!e^{-\gamma} \iint [\EE(S^2)-2\EE(S)^2] +24  \,e^\gamma \m \int[\EE(S^2)-2\EE(S)^2]\\
\hspace{0.7cm} + 96 \m \, e^\gamma\!\!\int \!\!e^{-\gamma} \iint [\EE(S)] +24  \,e^\gamma \m^2 \,\int [\EE(S)] 
-24 \m^3 \,e^\gamma\!\!\int \!\!e^{-\gamma}\\
\hspace{0.7cm}-8\,e^\gamma \int [-\EE(S^3)+6\EE(S^2)\EE(S)-6\EE(S)^3].
\end{array}
$$
Once more we compute
%\dot{\wideparen{e^{-\gamma} \dot Q}}
$$
\begin{array}{l}
(e^{-\gamma}Q')'=-48 \,e^{-\gamma} \iint [\EE(S^2)-2\EE(S)^2]-96 \, e^\gamma\!\!\int \!\!e^{-\gamma}\int [\EE(S^2)-2\EE(S)^2]\\
\hspace{1.8cm} - 24 \m e^{-\gamma} (\EE(\T^2)-2\EE(\T)^2)-192  \, e^\gamma (\int \!\!e^{-\gamma})^2 \iint [\EE(S)] -96 \m \,e^{-\gamma} \iint [\EE(S)]\\
\hspace{1.8cm} -192 \m \,e^\gamma\!\!\int \!\!e^{-\gamma} \int [\EE(S)]+144 \m^2 \,e^\gamma (\int \!\!e^{-\gamma})^2 +
8\,e^{-\gamma} (-\EE(\T^3)+6\EE(\T^2)\m-6\m^3).
\end{array}
$$
Let us compute the last term in this expression
$$
\begin{array}{l}
8\,e^{-\gamma}\left(-3\EE(\T^2)\m +12\m
\iint[\EE(S)]+6\iint[\EE(S^2)-2\EE(S)^2] +6\EE(\T^2)\m-6\m^3\right)\\ 
=8\,e^{-\gamma}\bigg(-3\EE(\T^2)\m +3\m(2\m^2-\EE(\T^2))+6\iint[\EE(S^2)-2\EE(S)^2] +6\EE(\T^2)\m-6\m^3\bigg)\\
=48\, e^{-\gamma} \iint[\EE(S^2)-2\EE(S)^2].
\end{array}
$$
Then, we conclude
$$
\begin{array}{l}
Z=\frac{(e^{-\gamma} Q')'}{e^\gamma\!\!\int \!\!e^{-\gamma}}\\
\;\;=-96 \int [\EE(S^2)-2\EE(S)^2]
-192  \int\!\! e^{-\gamma} \iint [\EE(S)]-192 \m
 \int [\EE(S)]+144  \m^2 \int \!\!e^{-\gamma}.
\end{array}
$$
The derivative of $Z$ is given by 
$$
\begin{array}{l}
Z'=96 e^{-\gamma} (\EE(\T^2)-2\EE(\T)^2)+192\, e^{-\gamma} \iint [\EE(S)]\\
\qquad+576\, e^\gamma\!\!\int \!\!e^{-\gamma} \int [\EE(S)]+48\m^2 e^{-\gamma}
-576 \m \, e^\gamma (\int \!\!e^{-\gamma})^2\\
\hspace{0.45cm}=48 \EE(\T^2)  \, e^{-\gamma} -48 \m^2 \, e^{-\gamma} + 576 \,e^\gamma\!\!\int \!\!e^{-\gamma} \int[\EE(S)]
-576 \m\, e^\gamma (\int \!\!e^{-\gamma})^2\\
\hspace{0.45cm}=48\, e^{-\gamma}  \hbox{Var}+  576 \,e^\gamma\!\!\int \!\!e^{-\gamma} \left(\int[\EE(S)]-\m\int e^{-\gamma}\right).
\end{array}
$$
Notice that 
$$
\frac{\int[\EE(S)]-\m\int e^{-\gamma}}{\int e^\gamma (\int \!\!e^{-\gamma})^2} \rightsquigarrow
\frac{-\m e^{-\gamma} +\m e^{-\gamma} +2 e^\gamma (\int \!\!e^{-\gamma})^2}{-e^\gamma (\int \!\!e^{-\gamma})^2}\to -2.
$$
Therefore, we get from Corollary \ref{cor:1} $(iii)$,
\ben
Z' &\rightsquigarrow& 48\, e^{-\gamma}  \hbox{Var}\, -1152 \,\left(e^{\gamma} \int \!\!e^{-\gamma}\right) 
\int e^\gamma \left(\int \!\!e^{-\gamma}\right)^2 \\
&&\rightsquigarrow 48\, e^{-\gamma}  \hbox{Var}\, -1152\, e^{3\gamma} \left(\int \!\!e^{-\gamma}\right)^4.
\een
On the other hand, we know that $(\hbox{Var})^2 \rightsquigarrow 64 \Psi^2$. So to study the asymptotic behavior
of the central fourth moment, it is enough to follow the same steps 
for $64 \Psi^2$. We get
$64 (\Psi^2)'=-128 \Psi \,e^\gamma\!\!\int \!\! e^\gamma (\int \!\!e^{-\gamma})^2$, which gives further that
$$
\begin{array}{l}
64\, (e^{-\gamma} (\Psi^2)')'=128\, e^\gamma( \int \!\! e^\gamma (\int \!\!e^{-\gamma})^2)^2 +
128 \Psi e^\gamma (\int \!\!e^{-\gamma})^2 \rightsquigarrow 128 \Psi e^\gamma (\int \!\!e^{-\gamma})^2.
\end{array}
$$
Call now $G$ this last quantity divided by $e^\gamma\!\!\int \!\!e^{-\gamma}$ and take a further derivative to obtain
$$
\begin{array}{l}
G'=-128 \Psi \, e^{-\gamma} - 128\,  e^\gamma \!\!\int \!\! e^{-\gamma}\, \int \!\! e^\gamma (\int \!\!e^{-\gamma})^2 
\rightsquigarrow -128 \Psi \, e^{-\gamma} \rightsquigarrow - 16\, \hbox{Var} \, e^{-\gamma} .
\end{array}
$$
Putting all these estimations together yield to
%\begin{equation}
%\label{eq:asym4moment}
%\begin{array}{l}
$$\frac{\EE_\infty((\T_z-\m)^4)}{(\hbox{Var}_\infty(\T_z))^2} \rightsquigarrow 3 +9 \frac{e^{4\gamma}\left(\int \!\!e^{-\gamma}\right)^4}{\Psi}
\rightsquigarrow 3 +9 \left(16 \q^4 \int \q^{-3}\right)^{-1}
\rightsquigarrow 3 +9 \frac{\q'}{4\q^2}\to 3,
$$
%\end{array} \end{equation}
which is the fourth moment of a N(0,1) distribution.
\end{proof}
\medskip

\section{Some basic spectral properties}
\label{App:C}

In this section we present the basic spectral properties of the semigroup associated to $X$ in the interval $[z,\infty)$, where
$z\ge 0$. 

\me
When the process is restricted to $[z,\infty)$, the speed measure is 
$$
\mu_z(dx)=2e^{-(\gamma(x)-\gamma(z))}\, dx,\quad x\in [z,\infty).
$$ 
The scale function is  given by  $\Lambda_z(x)=\int_z^x
e^{\gamma(y)-\gamma(z)}\, dy=e^{-\gamma(z)}(\Lambda(x)-\Lambda(z))$.

\me In what follows, we  denote by $\LL $ the second order
differential operator given by 
$$
\LL u=\frac12 u'' (x)-\q(x) u'(x),
$$
for all $u\in \mathcal{C}^2([z,\infty))$. On
  $\mathcal{C}^2_0([z,\infty))$ this operator is symmetric with
    respect to $\mu_z$ and it has a minimal 
closed symmetric extension on $L^2(\mu_z)$, which we shall denote by
$L_z$. This operator is the infinitesimal generator 
of the semigroup associated to the  process  $(X_{t}:\, t\ge 0)$  killed when it attains $z$ (see
\cite{Cattiaux}). 

The main result about the spectral decomposition of $L_z$ is Theorem
3.2 in \cite{Cattiaux}, which we summarize here. 

\begin{theorem}
\label{the:spectral} Assume $\HH_1$ and $\HH_2$ holds. Then, $-L_z$
has purely discrete spectrum $\{\lambda_{i}(z):\, i\ge 1\}$ 
that satisfies
\smallskip
\begin{itemize}
\item[(i)] $0<\lambda_{1}(z)<\lambda_{2}(z)<....$ is an increasing
  sequence of simple eigenvalues of $-L_z$; 
\item[(ii)] For every $\lambda_{i}(z)$ there exists an eigenfunction
  $\psi_{z,i}\in \mathcal{C}^2(z,\infty)$,  
unique up to a multiplicative constant, which also satisfies
$$
\begin{array}{l}
\LL \psi_{z,i}(\bullet)=-\lambda_{i}(z) \psi_{z,i}(\bullet),\, \psi_{z,i}(z)=0\\
\\
\int_z^\infty \psi^2_{z,i}(x)\, 2e^{-(\gamma(x)-\gamma(z))}\, dx=1.
\end{array}
$$
\end{itemize}
\smallskip
The set $\{\psi_{z,i}:\, i\ge 1\}$ forms an orthonormal basis of $L^2(\mu_z)$. 
Moreover, we can choose $\psi_{z,1}$ to be positive in $(z,\infty)$.
\end{theorem}

\medskip

The process $(X_{t\wedge \T_z}:\, t\ge 0)$ has a unique q.s.d. in
$[z,\infty)$ (see \cite{Cattiaux} Theorems 5.2 and 7.2),  
which is given by 
$$
\nu_z(dx)=\frac{\psi_{z,1}(x) e^{-\gamma(x)} \, dx}{\int_z^\infty \psi_{z,1}(y) e^{-\gamma(y)}\, dy} \,  \ind_{[z,\infty)}(x).
$$  
The fact there is a gap in the spectrum of $L_z$ implies that  $(X_{t\wedge \T_z}:\, t\ge 0)$ is $R$-positive, which
means that the associated $Q$-process is positive recurrent.

\medskip

For this article we need extra properties of the principal eigenvalue
and eigenfunctions. 

\medskip

\begin{proposition}\label{prop:b2} Under  $\HH_1$ and $\HH_2$ we have
\smallskip
\begin{itemize} 
\item[(i)] if $0\le z< x$ then $\lambda_{1}(z)<\lambda_{1}(x)$, that is the principal eigenvalue is 
a strictly increasing function. Moreover, for all $x\ge 0$
$$ 
\lambda_{1}(x)\ge \frac{(\inf\{q(y):\, y\ge x\}\vee 0)^2}{2}.
$$
In particular, $\lim\limits_{x\to \infty} \lambda_{1}(x)=\infty$.

\item[(ii)] For all $z\ge 0$ the functions $\psi_{z,i},\psi_{z,i}'$ are bounded, for every $i\ge 1$. 
The functions $\psi_{z,i},\psi_{z,i}'$ have exactly $i, i-1$ zeros respectively. $\psi_{z,i}$ is eventually monotone and 
the following limits exist
$$
\lim\limits_{x\to \infty} \psi_{z,i}(x)=\psi_{z,i}(\infty)\neq 0 \hbox{ and }
\lim\limits_{x\to \infty} \psi'_{z,i}(x)=0.
$$
\item[(iii)] For all $z\ge 0$, there exists a constant $A=A(z)>0$ such that $\lambda_{i}(z) \ge A i$ holds for all $i\ge 1$. 
\end{itemize}
\end{proposition}
\smallskip

\begin{proof} $(i)$ It is clear that $\lambda_{1}(z)\le \lambda_{1}(x)$ because for every $y>x$ we have
$$
\lambda_{1}(z)=\lim\limits_{t\to \infty} \frac{-\log \PP_y(\T_z>t)}{t}\le \lim\limits_{t\to \infty} \frac{-\log \PP_y(\T_x>t)}{t}=\lambda_{1}(x).
$$
The first equality in the above expression is well known (see for example Theorem \ref{the:3} or \cite{Cattiaux} Theorem 5.1).
Let us assume that $\lambda_{1}(z)=\lambda_{1}(x)$ and we will arrive to a contradiction. We denote by $\lambda=\lambda_{1}(z)$
and by $g=\psi_{z,1}, f=\psi_{x,1}$. Notice that, even if the eigenvalues are equal, these two functions are not because
$0<g(x), f(x)=0$.

\smallskip

Consider $W$ the Wronskian of these two functions $W=g'f-gf'$. The derivative of $W$ on $(x,\infty)$ is
\ben
W'(y)&=&g'' f-gf''=2(\q(y)g'(y)-\lambda g(y))f(y)-2(\q(y)f'(y)-\lambda f(y))g(y)\\
&=&
2q(y)W(y).
\een
In particular we have $W(y)=W(x) e^{\gamma(y)-\gamma(x)}=-g(x)f'(x)e^{\gamma(y)-\gamma(x)}$. 
Since $f$ is increasing we must have $f'(x)>0$ (otherwise $f\equiv 0$ since it is solution of a second order linear differential equation) and therefore $W(x)<0$.
Using that $(g/f)'=W/f^2$, we get for $y>y_0>x$
\be
0<\frac{g(y)}{f(y)}&=&\frac{g(y_0)}{f(y_0)}+W(x)\!\!\int_{y_0}^y \! \frac{e^{\gamma(w)-\gamma(x)}}{f^2(w)}\, dw \nonumber \\
&=&
\frac{g(y_0)}{f(y_0)}-g(x)f'(x)\!\!\int_{y_0}^y \!\frac{e^{\gamma(w)-\gamma(x)}}{f^2(w)}\, dw. \label{eq:-1}
\ee
The Cauchy-Schwarz inequality shows that
\ben
(y-y_0)^2&=&\left(\int_{y_0}^y \frac{f(w) e^{-\frac12 (\gamma(w)-\gamma(x))}}{f(w) e^{-\frac12 (\gamma(w)-\gamma(x))}}\, dw\right)^2\\
&\le&
 \int_{y_0}^y f(w)^2 e^{-(\gamma(w)-\gamma(x))}\, dw  \int_{y_0}^y \frac{e^{(\gamma(w)-\gamma(x))}}{f^2(w)}\, dw.
\een
Since  $\int_x^\infty f(w)^2 e^{-(\gamma(w)-\gamma(x))}\, dw<\infty$, we conclude that
$$
\int_{y_0}^\infty \frac{e^{(\gamma(w)-\gamma(x))}}{f^2(w)}\, dw=\infty.
$$
This is a contradiction with \eqref{eq:-1} and therefore
$\lambda_{1}(z)<\lambda_{1}(x)$.

We now prove the lower bound for $\lambda_{1}(x)$. Here we use a comparison with a diffusion with constant drift 
and idea developed for example in \cite{Martinez}. We consider
$M=\inf\{q(y):\, y\ge x\}$. 
If $M\le 0$ the result is obvious. So assume that $M>0$. We compare the process
$X$ with the process $Y$ which has constant drift $-M$, that is,
$$
Y_t=Y_0 +B_t-Mt.
$$
If $Y_0=X_0>x$ then we have $X_t\le Y_t$ for all $0\le t \le \T_x(X)$,
showing that 
$\PP_x(\T_x(X)>t)\le \PP_x(\T_x(Y)>t)$. In particular the exponential
rates of absorption at zero 
are comparable
$$
\lambda_{1}(x)\ge \lambda^Y_{1}(x)=\frac{M^2}{2},
$$
showing the desired property and also that $\lim\limits_{x\to \infty}
\lambda_{1}(x)=\infty$.

\medskip

$(ii)$ We first prove that $\LL$ is of limit point type at $\infty$,
that is, for some $\lambda \in \CC$ (equivalently for all $\lambda \in
\CC$)  
the equation $\LL f=-\lambda f$ has a solution which is not in
$L^2(\mu_z)$ near $\infty$, which means for all $x>z$ 
$$
\int_x^\infty f^2(y) e^{-\gamma(y)}\, dy=\infty.
$$
To show this property, it is enough to consider $\lambda=0$ and
$f(y)=\Lambda(y)$. Here we use an argument taken from \cite{Littin}. 
For $M>0$ and $x>z$ we consider
\ben
M&=&\int_x^{M+x} \ind_{e^{-\gamma(y)} \Lambda^2(y)>1}+\ind_{e^{-\gamma(y)} \Lambda^2(y)\le 1} \,dy \\
&\le &\int_x^{M+x} \Lambda^2(y) e^{-\gamma(y)}\, dy+  \int_x^{M+x}
\frac{e^{\gamma(y)}}{\Lambda^2(y)} \, dy\\ 
\\
&=&\int_x^{M+x} \Lambda^2(y) e^{-\gamma(y)}\, dy-\frac{1}{\Lambda(y)}\big|_{x}^{M+x}\le
\int_x^{M+x} \Lambda^2(y) e^{-\gamma(y)}\, dy+\frac{1}{\Lambda(x)}.
\een
This implies that $\int_x^\infty \Lambda^2(y) e^{-\gamma(y)}\,
dy=\infty$, showing the claim. 

\smallskip

We use this fact for $\lambda=\lambda_{i}(z)$. In what follows we shall
prove the existence of  
a nonzero bounded solution to the equation $\LL u=-\lambda u$ on
$[z,\infty)$. Notice that this equation has two 
linearly independent solutions and one of them is chosen to be
$\psi_{z,i}$. The other one can be 
chosen to be the solution of $\LL v=-\lambda v, v(z)=1,
v'(z)=0$. Since $\psi_{z,i}$ is in $L^2(\mu_z)$ near $\infty$ 
then necessarily $v$ cannot be in $L^2(\mu_z)$ near $\infty$. 

\smallskip

On the other hand, $u=a\psi_{z,i}+b v$ for some $a,b \in \RR$, but the
fact that $u$ is bounded implies $u$ is in $L^2(\mu_z)$ near
$\infty$. 
The conclusion is that $b=0$ and therefore $u=a\psi_{z,i}$ for some
$a\neq 0$, which implies that $\psi_{z,i}$ is bounded, showing 
the result. 

\smallskip

So, let us construct this bounded solution $u$. Here we follow the
ideas of \cite{Levinson}.  
Consider a large $x_0>z$ such that for all $x\ge x_0$ the following
properties hold  (these are consequences of our hypotheses $\HH_1,
\HH_2$) 
\begin{itemize}
\item[(1)] $\frac{\lambda+\lambda^2}{\q^2(x)}<\frac14$;
\smallskip

\item[(2)] $(\lambda+1) \left|\frac{\q'(x)}{\q^2(x)} \right|\le \frac14$;
\smallskip

\item[(3)] $\int_x^\infty \frac{1}{\q(y)} \, dy<\frac18$.
\end{itemize}
\smallskip

The function $S(x)=-\lambda \int_x^\infty \frac{1}{\q(y)} \, dy$ is
well defined and satisfies $S'(x)=\frac{\lambda}{\q(x)},
S''(x)=-\frac{\lambda\q'(x)}{\q^2(x)}$ 
and a fortiori we obtain
$$
| S''(x)+S'(x)^2|\le \frac12.
$$
Now, consider the integral $I(x)=\int_x^\infty e^{\gamma(y)-2S(y)}\int_y^\infty e^{-\gamma(\xi)+2S(\xi)}\, d\xi \,dy$. 
The fact that $S$ is bounded near $\infty$ and $\HH_1$ holds, show that $I(x)$ is finite.
We have
$$
I(x)=\int_x^\infty e^{\gamma(y)-2S(y)} \int_y^\infty \frac{-1}{2\q(\xi)-2\lambda/\q(\xi)} d(e^{-\gamma(\xi)+2S(\xi)}) \,dy.
$$
Integrating by parts the second integral we get
$$
\begin{array}{l}
I(x)=\int_x^\infty \frac{1}{2\q(y)-2\lambda/\q(y)}\, dy
+\frac12\int_x^\infty e^{\gamma(y)-2S(y)}\int_y^\infty \frac{\q'(\xi)+\lambda \q'(\xi)/\q^2(\xi)}{(\q(\xi)-\lambda/\q(\xi))^2} e^{-\gamma(\xi)+2S(\xi)}\, d\xi \,dy\\
\\
\le \int_x^\infty \frac{1}{\q(y)}\, dy 
+\frac12\int_x^\infty e^{\gamma(y)-2S(y)}\int_y^\infty \left|\frac{\q'(\xi)}{\q^2(\xi)}\right| 
\frac{1+\lambda/ \q^2(\xi)}{(1-\lambda/\q^2(\xi))^2} e^{-\gamma(\xi)+2S(\xi)}\, d\xi \,dy\\
\\
\le \frac18+ \frac12\int_x^\infty e^{\gamma(y)-2S(y)}\int_y^\infty e^{-\gamma(\xi)+2S(\xi)}\, d\xi \,dy.
\end{array}
$$ 
Thus, $I(x)\le 1/4$.
In the space $\mathcal{B}=\{u\in \mathcal{C}([x_0,\infty)): \, \lim\limits_{x\to \infty} u(x)=0\}$, we define the affine operator
$$
\mathscr{A} h(x)=-\int_x^\infty e^{\gamma(y)-2S(y)}\int_y^\infty [S''(\xi)+S'(\xi)^2][1+h(\xi)] e^{-\gamma(\xi)+2S(\xi)}\, d\xi \,dy.
$$
This operator is a contraction on the unit ball of $\mathcal{B}$. Thus, there exists a unique fixed point 
$\mathscr{A} h=h$ in the unit ball. It is straightforward to show that the function $u(x)=e^{S(x)}(1+h(x))$, for $x\ge x_0$, 
satisfies the equation $\LL u=-\lambda u$ on $[x_0,\infty)$ and it is bounded there. Now, extend this solution to the interval $[z,x_0+1]$ by solving the equation
$$
\LL \tilde u=-\lambda \tilde u, \, \tilde u(x_0+1)=u(x_0+1), \tilde u'(x_0+1)=u'(x_0+1).
$$
Gluing these functions together provides a $\mathcal{C}^2$ bounded solution of $\LL u=-\lambda u$ on $[z,\infty)$ as desired.

\medskip

The fact that each $\psi_{z,i}$ has exactly $i$ zeros is a consequence of the interlacing Cauchy Theorem. 
Recall that $\psi_{z,1}$ vanishes only at $x=z$.
Now, we show that each $\psi_{z,i}$ has a limit at infinity. We denote by $\lambda=\lambda_{i}(z)$ and
$\psi=\psi_{z,i}$. We also take $z=0$ to keep notation simple. Notice that
$(\psi'e^{-\gamma})'=-2\lambda \psi$, which implies that
$$
\psi(x)=\int_0^x e^{\gamma(y)} \big(\psi'(0)-2\lambda\int_0^y e^{-\gamma(\xi)}\psi(\xi) \,d\xi\big)\,dy.
$$
If $\psi'(0)\neq 2\lambda\int_0^\infty e^{-\gamma(\xi)} \psi(\xi) \,d\xi$ then $|\psi|$ grows like $\Lambda$ at infinity and therefore
$\psi$ is not in $L^2(d\mu)$, which is a contradiction. Therefore, we conclude
\begin{equation}
\begin{array}{l}
\psi(x)=2\lambda \int_0^x e^{\gamma(y)} \int_y^\infty e^{-\gamma(\xi)} \psi(\xi)\,d\xi\,dy, \label{psik}\\
\\
\psi'(x)=2\lambda \, e^{\gamma(x)}\int_x^\infty e^{-\gamma(\xi)} \psi(\xi)\,d\xi, \qquad 
\psi'(0)=2\lambda \int_0^\infty e^{-\gamma(\xi)} \psi(\xi)\,d\xi.
\end{array}
\end{equation}
For large values of $\xi$ the function $\psi$ has constant sign (recall that $\psi$ has a finite number of zeros). Then,
changing $\psi$ to $-\psi$, if necessary, we can assume that $\psi$ is eventually positive. Hence, $\psi'$ is eventually
positive and  $\psi$ is eventually increasing. Since it is bounded, we conclude that $\psi(\infty)>0$ exists.

That $\psi'$ converges to $0$ at infinity follows from Corollary \ref{cor:1} (i) and the fact that $\psi$ is bounded.
 On the other hand, if $\psi'(x)=0$ then either $\psi''(x)=-2\lambda \psi(x)< 0$ and $x$ is a local maximum for $\psi$ and $\psi(x)>0$ or
 $\psi''(x)=-2\lambda \psi(x)>0$ and $x$ is a local minimum for $\psi$ and $\psi(x)<0$. The case $\psi(x)=0$ is ruled out because
 in this case we would have $\psi\equiv 0$. The conclusion
 is that between two consecutive zeros of $\psi'$ there exists a unique zero of $\psi$. This together with the fact that $\psi$
 has $i$ zeros and it is eventually monotone, show that $\psi'$ has exactly $i-1$ zeros.

$(iii)$ The lower bound
for the eigenvalues is taken from \cite{Littin}, where it is shown that
$$
\lambda_{i}(z)\ge \frac{i-1}{C_z},
$$
with $C_z=2\int_z^\infty \Lambda(x) e^{-(\gamma(x)-\gamma(z))} \,dz=2\int_z^\infty \Lambda(x) e^{-\gamma(x)} \,dx$.
Using that $\lambda_{1}(z)>0$ and replacing $i-1$ by $i$ for $i\ge 2$, we conclude 
$$
\lambda_{i}(z)\ge A i,
$$
for some positive constant $A=A(z)$.
\noindent
\end{proof}
%%%%%%%%%%%%%%%%%%%%%%%
%%%%%%%%%%%%%%%%%%%%%%%%%

\bigskip Recall that $\lambda_k=\lambda_{k}(0)$ and $\psi_k=\psi_{0,k}$ for any integer $k\ge 1$. The estimation we need is given in the next result.

\medskip

\begin{proposition} 
\label{pro:bound4r}
Assume that $q$ satisfies $\HH_1,\HH_2, \HH_3$. Then, for every $\epsilon>0$ there exists
a finite constant $A=A(\epsilon)$ such that, for all $k\ge 1$
$$
\|\psi_k\|_\infty\le A e^{\epsilon \lambda_k}.
$$
\end{proposition}

\begin{proof} We recall that $(\psi_k)_k$ and $(\lambda_k)$ satisfy the following basic properties:
Each $\psi_k$ is a bounded continuous function with $k$ zeros,  $\psi_k$ is eventually increasing with a finite positive 
limit at infinity  $\psi_k(\infty)>0$, $\psi'_k$ is a bounded function, eventually positive with zero limit at infinity and 
$$
\begin{array}{l}
\psi_k''-2q\psi_k'=-2\lambda_k \psi_k,\, \psi_k(0)=0,\, \|\psi_k\|_{L^2}=1,\\
\\
\psi_k'(x)=2\lambda_k e^{\gamma(x)}\int_x^\infty e^{-\gamma(\xi)} \psi_{k}(\xi)\, d\xi,\, 
\psi_k(x)=2\lambda_k \int_0^x e^{\gamma(y)}\int_y^\infty e^{-\gamma(\xi)} \psi_{k}(\xi)\, d\xi\, dy.
\end{array}
$$
Also for some finite constant $B>0$ it holds $\lambda_k\ge B k$.
We also recall the function $\m(z)=\EE_\infty(\T_z)=2\int_0^\infty e^{\gamma(y)}\int_y^\infty e^{-\gamma(\xi)} d\xi\, dy$,
which is a continuous decreasing function that satisfies $\m(\infty)=0$ and
$$
\lambda_{1}(z)\m(z)>1.
$$
The first important thing to introduce is $\zeta_k$, the largest zero of $\psi_k$. Notice that $\zeta_1=0$. The function
$\psi_k$ is positive in $(\zeta_k,\infty)$ and satisfies
$\LL \psi=-\lambda_k \psi, \ \psi(\zeta_k)=0$ and 
$$
\int_{\zeta_k}^\infty 2e^{-(\gamma(x)-\gamma(\zeta_k))}\psi_k^2(x)\, dx\le
e^{\gamma(\zeta_k)} \int_0^\infty  2e^{-\gamma(x)}\psi_k^2(x)\,dx<\infty.
$$
The conclusion is that $\psi_k$ is proportional to $\psi_{\zeta_k,1}$ and $\lambda_k=\lambda_{1}(\zeta_k)$, that
is $\lambda_k$ is the exponential absorption rate for $X$ in $[\zeta_k,\infty)$. One important conclusion is
\begin{equation}
\label{eq:7.5}
\lambda_k \m(\zeta_k)>1.
\end{equation}
One can see that $(\zeta_k)_k$ is an increasing sequence. Let us show it converges to infinity. Indeed, if
this sequence is bounded, let us say by $z>0$, then we will obtain that $\lambda_k\le \lambda_{1}(z)<\infty$.
This is not possible, because $(\lambda_k)_k$ is the spectrum of the unbounded operator $\LL$.

\medskip

Now, we produce our first crude estimation on $\psi_k$. Notice that $q$ is eventually positive, so
$\gamma$ is eventually increasing. This gives, for some finite constant $C$ independent of $k,x$
\ben
|\psi_k(x)|&\le &C\lambda_k e^{\gamma(x)} \int_0^x \int_y^\infty e^{-\gamma(\xi)} |\psi_k(\xi)|\, d\xi dy\\
&\le &C\lambda_k e^{\gamma(x)} \int_0^x \left(\int_y^\infty e^{-\gamma(\xi)} \psi^2_k(\xi)\, d\xi\right)^{1/2}
\left(\int_y^\infty e^{-\gamma(\xi)}\, d\xi\right)^{1/2} dy\\
&\le& C \lambda_k e^{\gamma(x)} \int_0^x \left(\int_y^\infty e^{-\gamma(\xi)}\, d\xi\right)^{1/2} dy \\
&= &
C \lambda_k e^{\gamma(x)} \int_0^x e^{-\gamma(y)/2} \left(e^{\gamma(y)} \int_y^\infty e^{-\gamma(\xi)}\, d\xi\right)^{1/2} dy.
\een
Corollary \ref{cor:1} gives the asymptotic $e^{\gamma(y)} \int_y^\infty e^{-\gamma(\xi)}\, d\xi
\approx \frac{1}{2q(y)}$. Therefore
for some finite constants $D,E$ we have, for all $y\ge 0$
\begin{equation}
\label{eq:8}
e^{\gamma(y)} \int_y^\infty e^{-\gamma(\xi)}\, d\xi\le D \frac{1}{q(y)+E}.
\end{equation}
Back to our bound for $\psi_k$, we obtain
\ben
|\psi_k(x)|&\le &CD^{1/2} \lambda_k e^{\gamma(x)} \int_0^x \frac{e^{-\gamma(y)/2}}{\sqrt{q(y)+E}} dy \\
&\le& CD^{1/2} \lambda_k e^{\gamma(x)} \left(\int_0^\infty e^{-\gamma(y)}dy \int_0^\infty\frac{1}{q(y)+E} dy\right)^{1/2}.
\een
So far, we have for some finite $F$  the following bound holds
\begin{equation}
\label{eq:9}
|\psi_k(x)|\le F \lambda_k e^{\gamma(x)}.
\end{equation}

The next step is to use hypothesis $\HH_3$ to get the result. For that purpose we take a large constant $H$ (for the moment larger than 1), 
that will depend on $a$ given in hypothesis $\HH_3$ and that we will make explicit later. We now choose $x_1=x_1(a)\ge x_0$ such that the
following conditions hold 
\begin{align}
&\label{eq:10.1} q(x_1)\ge \max\{q(z): z\le x_1\}\vee 1;\\ 
&\label{eq:10.3} S(x_1):=\left(\frac{2H}{a^2}+\sqrt{H}\right) \int_{x_1}^\infty \frac{1}{q(y)} dy \le \frac{\epsilon}{3}.
\end{align}
For example $x_1=\inf\{x\ge x_0: q(x)\ge \max\{q(z):\, z\le x_0\}\vee 1+1, S(x)\le \frac{\epsilon}{3}\}$ will work. 
The next step is to find   $\chi_k$ such that $\zeta_k< \chi_k$ and
$$
q(\chi_k)=\sqrt{H \lambda_k} \,.
$$
For that purpose, we notice that for $x\ge x_1$ we have 
$$
\underline\lambda(x)\ge \frac12\left(\inf\{q(z):\, z\ge x\}\right)^2 \ge \frac{a^2}{2} q^2(x).
$$
In particular, $\sqrt{\lambda_k}\ge \frac{a}{\sqrt{2}} q(\zeta_k)$. So, if $H>\frac{2}{a^2}$ we deduce
that
$$
\chi_k:=\inf\{x>\zeta_k:\, q(x)\ge \sqrt{H \lambda_k}\} >\zeta_k.
$$

Now, we estimate $\psi_k(y)$ for $0\le y\le \chi_k$. This will done under the assumption that
$x_1<\chi_k$. Notice that since $\zeta_k<\chi_k$, we have $x_1<\chi_k$ for all large $k\ge k_0$.
According to our basic estimation we need to bound $\gamma$
in this interval. For $x_1\le y \le \chi_k$, we have
$$
\gamma(y)= \gamma(x_1)+\int_{x_1}^{y} 2q(z)dz\le 
\gamma(x_1)+\int_{x_1}^{\chi_k} 2q(z)dz=
\gamma(x_1)+\int_{x_1}^{\chi_k} 2\frac{q^2(z)}{q(z)}dz.
$$
Using  $\HH_3$ if $x_1\le z \le \chi_k$, we have $q(\chi_k)\ge \inf\{q(x);\, x\ge z\}\ge a q(z)$ and therefore
$$
\gamma(y)\le \gamma(x_1)+\frac{2}{a^2}q^2(\chi_k)\int_{x_1}^\infty \frac{1}{q(y)} dy=
 \gamma(x_1)+\frac{2}{a^2}H\lambda_k \int_{x_1}^\infty \frac{1}{q(y)} dy.
$$
This gives the desired bound for all $k\ge k_0$ and $0\le y\le  \chi_k$
$$
|\psi_k(y)|\le Fe^{\max\{\gamma(z):\, z\le x_1\}} \lambda_k \exp\left(\frac{\epsilon}{3} \lambda_k\right).
$$

The final step is to estimate $\psi_k(y)$ on the interval $[\chi_k,\infty)$. This is done by estimating the ratio
$R=\psi_k'/\psi_k$. Assume there exists $y$ in this interval one has  $R(y)>\sqrt{H}\lambda_k/q(y)$. Since $R(z)$ converges to $0$
as $z\to \infty$, we can find a first value $y<z$ such that
$$
R(z)=\frac{\sqrt{H}\, \lambda_k }{q(y)}.
$$
Obviously, we have $R'(z)\le 0$. On the other hand
\ben
0\ge R'(z)\psi_k^2(z)&=&\psi_k''(z)\psi_k(z)-(\psi_k'(z))^2\\
&=&(2q(z)\psi_k'(z)-2\lambda_k \psi_k(z))\psi_k(z)-(\psi_k'(z))^2\\
&=&2q(z)\frac{\sqrt{H}\, \lambda_k}{q(y)}\psi_k^2(z)-2\lambda_k \psi_k^2(z)-(\frac{\sqrt{H}\,\lambda_k}{q(y)})^2\psi_k^2(z).
\een
Since $\psi_k(z)>0$, we get
$$
2q(z)\frac{\sqrt{H}\, \lambda_k}{q(y)}\le 2\lambda_k +\left(\frac{\sqrt{H}\,\lambda_k}{q(y)}\right)^2.
$$
Recall that $q(y)\ge a q(\chi_k)$ and therefore $\left(\frac{\sqrt{H}\, \lambda_k}{q(y)}\right)^2\le \frac{1}{a^2} \lambda_k$.
In other words, we get the inequality
$$
aq(y)\le q(z)\le \frac{(2+1/a^2)}{2\sqrt{H}} q(y)
$$
which is not possible if $H> H_0=:\left(\frac{2+1/a^2}{2a}\right)^2>\frac{2}{a^2}$. Hence,
$$
\frac{\psi_k'(y)}{\psi_k(y)}\le \frac{\sqrt{H}\,\lambda_k}{q(y)},
$$
for all $y\ge \chi_k$ and a fortiori 
\ben
\psi_k(y)&=& \psi_k(\chi_k)\exp\left(\int_{\chi_k}^y R(z) dz\right) \\
&& \le \psi_k(\chi_k) \exp\left(\sqrt{H}\, \lambda_k \int_{x_1}^\infty \frac{1}{q(z)}dz\right)\le \psi_k(\chi_k) \exp\left(\frac{\epsilon}{3}\lambda_k\right).
\een
Putting these pieces together, we get for all $k\ge k_0$ and all $y\in [0,\infty]$,
$$
|\psi_k(y)|\le Fe^{\max\{\gamma(z):\, z\le x_1\}} \lambda_k \exp\left(\frac{2\epsilon}{3} \lambda_k\right).
$$
Since $x\le e^{x}$, we may change this inequality to
$$
|\psi_k(y)|\le \frac{3}{\epsilon} Fe^{\max\{\gamma(z):\, z\le x_1\}} \exp\left(\epsilon \lambda_k\right).
$$
To include the finite number of functions $\psi_k: k\le k_0$ we define
$$
A=\max\{\|\psi_k\|_\infty:\, k\le k_0\}\vee  \frac{3}{\epsilon} Fe^{\max\{\gamma(z):\, z\le x_1\}},
$$
to get the desired bound, for all $k\ge 1$, $
\|\psi_k\|_\infty \le A \exp\left(\epsilon \lambda_k\right).$
\end{proof}

\begin{corollary} 
\label{cor:2}
Assume that $q$ satisfies $\HH_1,\HH_2, \HH_3$. There are finite finite constants $A,B$ such that
for all $k\ge 1, x\ge 0$
$$
|\psi_k'(x)|\le A\|\psi_k\|_\infty \lambda_k \frac{1}{q(x)+B}.
$$
In particular, for all $w\ge z\ge 0$ (including $w=\infty$)
$$
\sup_{w\ge x\vee y\ge  x\wedge y \ge z} |\psi_k(x)-\psi_k(y)|\le A\|\psi_k\|_\infty \lambda_k \int_z^w \frac{1}{q(u)+B} du.
$$
Finally, for all $\epsilon>0$ there exist $C=C(\epsilon), z_0=z_0(\epsilon)$ such that, for all $z\ge z_0, k\ge 1$
$$
\sup_{x\wedge y\ge z} |\psi_k(x)-\psi_k(y)|\le C\,  e^{\epsilon \lambda_k}.
$$

\end{corollary}
\begin{proof} The result follows from the fact
$$
|\psi'_k(x)|\le 2 \|\psi_k\|_\infty e^{\gamma(x)}\int_x^\infty e^{-\gamma(y)} \, dy\le
 A\|\psi_k\|_\infty \lambda_k \frac{1}{q(x)+B},
 $$
 for some finite constants $A,B$. The rest of the proof is a consequence of 
 the previous proposition and the integrability of $1/q$ near infinity.
\end{proof}

\section*{Acknowledgments}

This work was partially funded by the Chair "Mod\'elisation Math\'ematique et Biodiversit\'e" of 
VEOLIA-Ecole Polytechni\-que-MnHn-FX and by the  CMM Basal project PBF-03 .

\end{document}